\documentclass[12pt]{amsart}
\usepackage{amssymb,amscd, amsfonts, amsxtra}
\usepackage{amsmath, amsthm}
\usepackage{psfrag}
\usepackage{amsrefs}
\usepackage{graphicx,epstopdf}
\usepackage{verbatim}
\usepackage{tikz-cd}

\usepackage[margin=3.5cm]{geometry}

\newcounter{SideCounter}

\newtheorem*{theorem*}{Theorem}

\let\from\colon

\def\S{\mathbb S}
\def\C{\mathbb C}
\def\Z{\mathbb Z}
\def\chat{\hat\C}
\def\M2{{\mathcal M}^{cm}_2}
\def\tp{\tilde{p}}
\def\tq{\tilde{q}}
\def\qg{\mathcal{Q}_\Gamma}
\let\goesto\to

\begin{document}

\begin{center}
{\bf On deformation spaces of quadratic rational functions}
\end{center}

\begin{center}
{\it  Tanya Firsova\footnote{Research of the first author was supported in part by NSF.}, Jeremy Kahn \footnote{Research of the second author was supported in part by NSF.}, Nikita Selinger}
\end{center}

\begin{center}
{\bf Abstract}
\end{center}
We study the group of self-equivalences of a partially postcritically finite branched cover and show  that certain deformation spaces of rational maps are not contractible.

\newtheorem{lemma}{Lemma}[section]
\newtheorem{definition}[lemma]{Definition}
\newtheorem{theorem}[lemma]{Theorem}
\newtheorem{note}[lemma]{Note}
\newtheorem{corollary}[lemma]{Corollary}
\newtheorem{question}[lemma]{Question}
\newcommand{\Teich}{\operatorname{Teich}}
\newcommand{\ATeich}{\operatorname{\overline{Teich}}}
\newcommand{\Per}{\operatorname{Per}}
\newcommand{\Def}{\operatorname{Def}}
\newcommand{\Mod}{\operatorname{Mod}}
\newcommand{\CC}{\mathbb{C}}
\newcommand{\PP}{\mathbb{P}^1}
\newcommand{\Rat}{\operatorname{Rat}}
\newcommand{\PSL}{\operatorname{PSL}}
\newcommand{\Map}{\operatorname{MCG}}
\newcommand{\Ker}{\operatorname{Ker}}
\newcommand{\Push}{\operatorname{Push}}
\newcommand{\id}{\operatorname{id}}

\section{Introduction} 

The introduction of complex methods to $1$-dimensional dynamics initiated a thorough study of parameter spaces of dynamical systems. The structure of parameter spaces, which turned out to be closely related to the structure of dynamical systems themselves, yields a better understanding of how perturbations of dynamical systems affect the dynamics. Higher dimensional parameter spaces, such as parameter space of polynomials of degree at least 3  or rational functions of degree at least 2, are significantly harder to visualize and study. It is natural to consider parameter spaces of rational functions with some prescribed dynamical properties; these are curves (algebraic varieties) defined by certain postcritical relations.  

In \cite{Milnor_quad_rational}, J. Milnor studied the space ${\mathcal M}_2$ of equivalence classes of quadratic rational maps with marked critical points. The maps are equivalent if they are conjugate by a M\"{o}bius transformation. He showed that ${\mathcal M}_2$ is biholomorphic to ${\mathbb C}^2$. Each representative $f\in {\mathcal M}_2$ has two critical points. We denote them by $c_1, c_2$. J.Milnor identified and studied the following $1$-dimensional curves:
 
$$\Per_n=\{f\in {\mathcal M}_2:\, f^k(c_1)=c_1\,  \mbox{if and only if}\, n|k \}.$$ 


The deformation space approach to the study of varieties defined by postcritically finite relations was suggested by A.Epstein. 
Inspired by the work of W.Thurston on postcritically finite maps, he introduced deformation spaces into holomorphic dynamics \cite{Epstein_FS}, \cite{Epstein_transversality}.  The cornerstone of W.Thurston's approach to postcritically finite maps is the pull-back map on the Teichm\"{u}ller space induced by the branched cover. A.Epstein used the pull-back map to define deformation spaces.

In this preprint we study topological properties of the deformation spaces. In Section \ref{sec:background} we recall the necessary background. In Section~\ref{sec:Def} we define deformation spaces. In Section \ref{sec:groups} we define a group of self-equivalences, a natural generalization of centralizers for obstructed postcriticaly finite branched covers \cite{BD}. These are subgroups of the mapping class group naturally associated to deformation spaces.  We introduce equalizing and arithmetically equalizing multicurves, natural analogs of W.Thurston invariant multicurves in Section \ref{sec:multicurves}. We give conditions on the accumulation of deformation spaces to the strata in the augmented Teichm\"{u}ller space in terms of arithmetically equalizing curves. In particular, we establish the following.

\begin{theorem*}  $f^*\mathcal{S}_{\Gamma}=i_*\mathcal{S}_{\Gamma}$ if and only if $\Gamma$ is an equalizing multicurve. If the deformation space accumulates on $\mathcal{S}_{\Gamma}$, then  there exists $\tilde{\Gamma}\subset \Gamma$ such that $\tilde{\Gamma}$ is an arithmetically equalizing multicurve.
\end{theorem*}

 We also show that each arithmetically equalizing multicurve give rise to a self-equivalence. In Section \ref{sec:per4} we study the example of $\Per'_4$ deformation space. 
 
A celebrated theorem of W. Thurston provides a topological criterion for when a postcritically finite branched covering is equivalent to a rational map. We give conditions that guarantee that a deformation space is not empty (see Section \ref{sec:vanishing}).

A. Epstein used deformation spaces to prove transversality results in holomorphic dynamics \cite{Epstein_transversality}. In particular, he proved  \cite{Epstein_transversality}\label{thm:Adam}  that if $f$ is not exceptional, i.e. $f$ is not a flexible Lattes example, then $\Def(f,A,B)$ is either empty or a smooth submanifold of $\Teich(B)$ of dimension $\#B-\#A$.


He asked whether  deformation spaces are contractible. 
In Section \ref{sec:noncontractible} we give a simple condition that guarantees that the deformation
space is not contractible. We prove that  deformation spaces corresponding to maps in $\Per'_n$, $n\geq 4$  satisfy
this condition. 

\begin{theorem*} There exists $f\in \Per'_n, n\ge 4$ such that $\Def(f,A,B)$  is not contractible.
\end{theorem*}

E. Hironaka and S. Koch proved that deformation space for $\Per'_4$ maps is not connected \cite{HK}. See also \cite{Rees} for relevant results.

In the Appendix we give a classification of equalizing curves for $\Per'_4$ deformation space. 
\subsection{Acknowledgements} We would like to thank Jan Kiwi for his valuable suggestions. 

\section{Background. Notations.}\label{sec:background}

Let $P \subset \S^2$ be a finite subset. 
First we summaries the notations in the following table, then we recall the definitions.  

\vspace{0.2cm}
\noindent\begin{tabular}{| c | l |}

\hline

$\Teich(P)$ &  Teichm\"{u}ller space of the sphere with the marked set $P$\\

\hline

$\Mod(P)$ & Moduli space of the sphere with the marked set $P$ \\
\hline

$\Map(P)$ & (pure) mapping class group of $\Teich(P)$\\
\hline

$\pi:\Teich(P)\to \Mod(P)$ & the natural projection from $\Teich(P)$ to $\Mod(P)$\\
\hline

${\mathcal C}(P)$ & the space of equivalence classes of simple closed curves\\ 
& on $({\mathbb S}^2, P)$ \\
\hline

${\mathcal M}(P)$ & the space of equivalence classes of multicurves \\ 
& on $({\mathbb S}^2, P)$\\
\hline

${\mathcal W}(P)$ & the weighted multicurves on $({\mathbb S}^2, P)$\\
\hline

$S_{\Gamma}$ & the stratum in the augmented Teichm\"{u}lller space \\
& obtained by pinching the multicurve $\Gamma$\\
\hline
\end{tabular}
\vspace{0.2cm}

An element of $\Teich(P)$ is an equivalence class of homeomorphisms $\phi:({\mathbb S}^2, B)\to (\chat, \phi(B))$. Two such homeomorphisms $\phi_1$ and $\phi_2$ are equivalent if there exist a homeomorphism $\beta:({\mathbb S}^2, P)\to ({\mathbb S}^2, P)$, homotopic to the identity, and a M\"{o}bius transformation $\alpha$, such that the following diagram commutes. 

\begin{tikzcd}
({\mathbb S}^2, P) \arrow{r}{\phi_1} \arrow{d}{\beta}  & (\chat, \phi_1(P)) \arrow{d}{\alpha}  \\
({\mathbb S}^2, P) \arrow{r}{\phi_2} &  (\chat, \phi_2(P))
\end{tikzcd}

In our setting, the Teichm\"{u}ller space is the universal cover of the moduli space. The mapping class group $\Map(P)$ is naturally identified with the group of deck transformations of the cover $\pi$.


%
%




Consider non-trivial, non-peripheral, simple closed curves on $({\mathbb S}^2\setminus P)$. Let $\mathcal{C}(P)$ be the set of equivalence classes of such curves up to free homotopy relative to $P$. We say that two distinct equivalence classes do not intersect if they have disjoint representatives. 
 We denote the class of $\gamma\in ({\mathbb S}^2, P)$ by $[\gamma]_P$.  We say that $\Gamma=\{\gamma_1,\dots, \gamma_n\}$ is {\it a multicurve} on $({\mathbb S}^2,  P)$ if the curves $\gamma_i$ represent mutually disjoint equivalence classes. Two multicurves $\Gamma$, $\tilde{\Gamma}$ are equivalent if for each $\gamma\in \Gamma$ there exists $\tilde{\gamma}\in \tilde{\Gamma}$ such that $[\gamma]_P=[\tilde{\gamma}]_P$ and vice  versa. We denote by $\mathcal{M}(P)$ the set of equivalence classes of multicurves.


Let $\Gamma$ be a multicurve on $({\mathbb S}^2,  P)$; we denote by $({\mathbb S}^2,  P)/\Gamma$ the topological space with marked points obtained by collapsing each curve in $\Gamma$ to a point. This space, a nodal sphere, is a collection of spheres with marked points attached to each other at special points that we will call \emph{nodes}. Note that since every curve in $\Gamma$ is non-peripheral and no two curves in $\Gamma$ are homotopic to each other, there are at least three special points (either marked points, or nodes) on every sphere in the collection.

We will call the  Teichm\"{u}ller space of the nodal sphere $({\mathbb S}^2,  P)/\Gamma$, defined as the product of  Teichm\"{u}ller spaces of components of that sphere, \emph{the stratum} $\mathcal{S}_\Gamma$ corresponding to the (equivalence class of) multicurve $\Gamma$. Taking the union of all strata we obtain the augmented Teichm\"{u}ller space: $$\ATeich(P):=\bigcup_{\Gamma \in \mathcal{M}(P)} \mathcal{S}_\Gamma.$$ Note that, by definition, $\Teich(P)=\mathcal{S}_\emptyset \subset \ATeich(P)$.

A point on a stratum $\mathcal{S}_\Gamma$ can be represented by a continuous map $\phi$ from $({\mathbb S}^2,  P)$ onto a nodal Riemann surface with marked points which is surjective and injective except that each curve in $\Gamma$ is collapsed to a node. The equivalence relation on such maps is a straightforward generalization of the equivalence relation on homeomorphisms defined above.

The geodesic length functions $l_\gamma$ for $\gamma \in \mathcal{C}(P)$ extend to $\ATeich(P)$ by setting  $l_\gamma=0$ if $\gamma$ is homotopic to a node, $l_\gamma=\infty$ if $\gamma$ has non-trivial intersection with at least one of the nodes. In the remaning case, $\gamma$ is a curve on one of the components of the nodal surface and we measure $l_\gamma$ with respect to the hyperbolic metric of that component. The topology of the augmented Teichm\"{u}ller space can be defined by requiring that $\tau_k \to \tau$ if and only if $l_\gamma(\tau_k) \to l_\gamma(\tau)$ for all $\gamma \in \mathcal{C}(P)$. 

\section{Deformation spaces}
\label{sec:Def}
Recall that following J. Milnor we define the space of all quadratic rational maps with a critical orbit of period $n$:
$$\Per_n=\{f\in {\mathcal M}_2:\, f^k(c_1)=c_1\,  \mbox{if and only if}\, n|k \}.$$ 

Let $\tilde{A}=\{c_1, f(c_1), \dots f^{\circ (n-1)}(c_1)\}$. 
For all but finitely many classes $f\in \Per_n$, the second critical value $f(c_2)$ does not enter the critical cycle.
Let $$\Per_n'=\{f\in \Per_n:\ f(c_2)\not \in \tilde{A}\}.$$
Let $\tilde{B}=\tilde{A}\cup \{f(c_2)\}.$
Let $\Mod_{\tilde{B}}$ be the modular space of the sphere with the set $\tilde{B}$ marked. There is a natural embedding of embedding of $\Per'_n$ curve into $\Mod_{\tilde{B}}.$
The map $$j: \Per'_n\to \Mod_{\tilde{B}},$$ $$j(f)=(c_1, f(c_1), \dots, f^{n-1}(c_1), f(c_2))$$ is an injection. 

There is a natural cover $\Def(f,\tilde{A}, \tilde{B})$ of $j(\Per'_n)$ in the Teichm\"uller space $\Teich_{\tilde{B}}$ of the sphere with the set $\tilde{B}$ marked, which is an example of a deformation space. Let us now introduce a general definition of deformation spaces.

Let $f:{\mathbb S}^2\to {\mathbb S}^2$ be a branched cover, $V(f)$ be the  set of critical values of $f$. Let $A, B\subset {\mathbb S}^2$ be finite subsets such that 
$$A\cup f(A)\cup V(f)\subset B.$$
Since $V(f)\subset B$, the map $f:({\mathbb S}^2, f^{-1}(B))\to ({\mathbb S}^2, B)$ is a covering map. By changing the complex structure on $({\mathbb S}^2, B)$ and applying the pull-back, one obtains a complex structure on $({\mathbb S}^2, f^{-1}(B))$. Then by erasing points in $f^{-1}(B)\backslash A$ the complex structure on $({\mathbb S}^2, A)$. This map induces the pull-back map 
 $$f^*: \Teich(B) \to \Teich(A),$$
 where $\Teich (A)$, $\Teich (B)$ are Teichm\"{u}ller spaces of the sphere with sets $A$, $B$ marked.
 Since $A\subset B$, the forgetful map is well-defined:
 $$i_*\Teich( B)\to \Teich(A).$$

Following A. Epstein \cite{Epstein_FS}, \cite{Epstein_transversality} we define the deformation space $\Def(f,A,B)$: 
$$\Def(f,A,B):=\{\tau\in \Teich ({\mathbb S}^2, B) :\quad f^*\tau = i_*\tau\}.$$

Let $f\in \Per'_n$, by taking $A=\tilde{A}$, $B=\tilde{B},$ we obtain the deformation space $\Def(f, A,B)$. 

The deformation space is non-empty if and only if $f$ is equivalent to a rational map, 
that is, there exist a rational map $g\from \chat \to \chat$ and homeomorphisms $\phi, \psi\from \S^2\to \chat$ such that $\phi \sim \psi$ rel $A$ and the following diagram commutes:
$$\begin{CD}
(\S^2, A) @>\psi>> (\chat,\psi(A))\\
@VVfV @VVgV\\
(\S^2, B) @>\phi>> (\chat,\phi(B)).
\end{CD}$$

In fact, each structure $\tau\in\Def(f, A,B)$ defines such a rational map up to conjugacy by  M\"{o}bius transformations.


\begin{theorem} \label{te:aug_teich} The map $f^*$ extends continuously to the augmented Teichm\"{u}ller space.
\end{theorem}
\begin{proof} The proof of Theorem 6.4 in \cite{Selinger} applies mutatis mutandis to our setting.
\end{proof}

We say that a triple $(f,A,B)$ and, correspondingly, the deformation space $\Def(f,A,B)$ are {\it reduced} if 
$$B=A\cup f(A)\cup V(f).$$
Below we  assume that all  
 deformation spaces $\Def(f,A,B)$ we consider are reduced. 
Throughout the article, we consider only non-exceptional branched covers (the covers to which A.Epstein's theorem \ref{thm:Adam} is applicable). 

\section{Supgroups of mapping class groups naturally associated to deformation spaces} \label{sec:groups}

\subsection{Postcritically finite branched covers}
Let $f$ be a branched cover. 
Let $\Omega_f=\{z\in \hat{\mathbb C}: \deg_z f \geq 2 \}$. The set 
$$P_f=\bigcup_{n\geq 1} \{f^{\circ n}(\Omega_f\}$$
is called the postcritical set of $f$. The branched cover is called postcritically finite if $P_f$ is a finite set. Two branched covers $f$, $g$ are combinatorially equivalent \cite{DH_Thurston} if there exist homeomorphisms $h$, $h'$ isotopic rel $ P_f$ such that 
$$\begin{CD}
(\S^2, P_f) @>h'>> (\S^2,P_g)\\
@VVfV @VVgV\\
(\S^2, P_f) @>h>> (\S^2,P_g).
\end{CD}$$

W.Thurston  \cite{DH_Thurston} have shown that if the map $f$ has hyperbolic orbifold, then it is equivalent to at most one rational map; moreover he proved a topological criterion that checks whether a given map is equivalent to a rational map.

If in the diagram above we take $f=g$, then we say that $[h]\in \Map(P_f)$ is a self-equivalence. In particular, it follows from W.Thurston's theorem \cite{DH_Thurston} that if  $f$ is equivalent to a rational map, then  the group of self-equivalences is trivial. However, if $f$ is an obstructed branched cover, the group of self-equivalences can be quite complicated and has been extensively studied \cite{BD,KSY}. In particular, it can be infinitely generated \cite{BD}.  

  
\newcommand{\R}{\mathcal R}
\newcommand{\RC}{\R_{\Rat}}
\newcommand{\piRM}{\pi_{\mathcal{RM}}}
\newcommand{\piTR}{\pi_{\mathcal{TR}}}
\newcommand{\Deck}{\operatorname{Deck}}
\subsection{Self-equivalences for partially postcritically finite maps}
In this section we define a group of self-equivalences associated with a branched cover $f:(\S^2, A)\to (\S^2, B)$ such that $$B=A\cup f(A)\cup V(f).$$ 


Suppose, for $i = 0, 1$, that $f_i\from (\mathbb{S}^2, A_i) \to (\mathbb{S}^2, B_i)$ are branched covers, with $A_i \cup V(f_i) \subset B_i$. We say that a homeomorphism $h \from (\S^2, A_0, B_0) \to (\S^2, A_1, B_1)$  is an \emph{equivalence} (or \emph{combinatorial equivalence}) if there is a $h'\from(\mathbb{S}^2, A_0) \to (\mathbb{S}^2, A_1)$ such that $h' \sim h$ rel $A_1$ and the diagram commutes 
$$\begin{CD}
(\S^2, A_1) @>h'>> (\S^2,A_2)\\
@VVf_1V @VVf_2V\\
(\S^2, B_1) @>h>> (\S^2,B_2).
\end{CD}.$$
 We observe that this is Thurston's notion of equivalence in the case where $A_i = B_i$. 

If we take $f_0=f_1=:f$, then for each combinatorial equivalence $h:(\S^2, B)\to (\S^2, B)$, we can consider $[h]\in\Map(B)$.  We obtain a group $G_f\subset \Map(B)$ that we call {\it a group of self-equivalences} of the triple $(f, A,B)$. 

\subsection{Rees component}

In this section we define spaces naturally associated to the group of self-equivalences.
First we define a Rees component $R(f)$: 

We replace $\mathbb{S}^2$ with $\chat$ in the definition of combinatorial equivalence. We do not suppose that the $f_i$ are holomorphic. We say that an equivalence $h\from (\chat, A_0, B_0) \to (\chat, A_1, B_1)$ is \emph{geometric} if it is conformal (we do not require $h'$ to be conformal). 

We define the \emph{Rees component} $\R(f)$ as the space of branched covers $g\from (\chat, A') \to (\chat, B')$ that are combinatorially equivalent to $f$, modulo geometric equivalence. The space $\R(f)$ is a connected component of the space of all branched covers with the same critical portrait as $f$, modulo the same equivalence relation.
M.Rees in \cite{Rees} considered the space of all branched covers combinatorially equivalent to given one for quadratic rational maps, and kept track of the points $A'$ and $B'$. Rees component is a deformation retract of the space of branched covers that M.Rees considered. Note that the space is naturally equipped with a complex structure.

We let $\RC(f,A,B) \subset \R(f)$ be the equivalence classes with  holomorphic representatives. 
\begin{lemma} There is a natural covering map 

\begin{equation*} \piTR\from\Teich(B) \to \R(f). \end{equation*}
The group $G_f$ can be naturally identified with $\Deck(\Teich(B)/\R(f))$.

\end{lemma}

\proof
For $\tau\in \Teich(B)$, let $\phi\from (S^2, B) \to (\chat, \phi(B))$ be a representative, then we define $\piTR(\tau)=\phi \circ f\circ \phi^{-1}$. Choosing a different representative $\phi'$ of the structure $\tau$,  we 
obtain the  map $\phi'\circ f\circ (\phi')^{-1}$ which is geometrically equivalent to $\phi \circ f\circ \phi^{-1}$; we see that the map $\piTR$ is well-defined.

The map $\piTR$ is surjective. Let $g\from (\chat, A') \to (\chat, B')$ be a branched cover combinatorially equivalent to $f$, then there exists $\phi\from(\S^2, B)\to (\chat, B')$, $\phi'\from(\S^2, A)\to (\chat, A')$ such that $g\circ\phi'=\phi\circ f$, $\phi\sim \phi'$ rel $A$. The branched cover $g$ is thus  geometrically equivalent to $\phi \circ f\circ\phi^{-1}$.   

Moreover, if $h\in\Map(S^2, B)$ is a self-equivalence of $f$, then  $\piTR(h \tau)=\piTR(\tau)$.  The map $\piTR \colon \Teich(B) \to \mathcal{R}(f)$ is a covering map. It is a Galois cover, since $\Teich(B)$ is simply connected. We can identify $G_f:=\Deck(\Teich(B)/\R(f))$ with the group of homotopy classes of self-equivalences of $(f, A, B)$. 
\qed

We say that an element $h\in \Map(B)$ is {\it liftable} if there exists $h'\in \Map(A)$ \and their representatives $\tilde{h}$ and $\tilde{h}'$ such that $f\circ \tilde{h}=\tilde{h}'\circ f$. If such an element $h'\in \Map(A)$ exists, then it is unique. Hence, $f$ acts on liftable elements, we will use the notation $f^*\colon H_f \to \Map(A)$ for the action $h \mapsto h'$. The group of liftable elements $H_f<\Map(B)$ is a finite index subgroup \cite{KPS}.  The group $G_f$ is a subgroup of $H_f$.

Note that the group of self-equivalences $G_f$ and the Rees component $\R(f)$ are well-defined even if the branched cover $f$ is not realized by a rational function. Now we assume that $f$ can be realized by a rational function, i.e. $\RC(f,A,B)$ and $\Def(f, A,B)$ are non-empty.

\begin{lemma} The map $\piTR \colon \Def(f,A,B)\to\RC(f,A,B)$ is a Galois covering map.
\end{lemma}

\begin{proof} Let $\tau\in \Def(f,A,B)$; let $\phi:(\S^2, B)\to (\chat, \phi(B))$ be a representative of $\tau$. Then by  definition of the deformation space, there exists $\phi':(\S^2, A)\to (\chat, \phi(A))$, $\phi'\sim \phi$ rel $A$, and a rational map $g\from(\chat, \phi(A))\to(\chat, \phi(B))$ such that $g\circ \phi'=\phi \circ f$. Let $h$ be a self-equivalence and $\tau'\in \Teich(B)$ a structure defined by $\phi \circ h$. We have the following diagram:

\begin{tikzcd}
(\S^2,A)\arrow{r}{h'}\arrow{d}{f} & (S^2, A)\arrow{r}{\phi'}\arrow{d}{f} & (\chat, \phi(A))\arrow{d}{g}\\
(\S^2,B)\arrow{r}{h} & (S^2, B)\arrow{r}{\phi} & (\chat, \phi(B))\\
\end{tikzcd}

Hence, $\tau' \in \Def(f,A,B)$. So, the group $G_f$ acts on $\Def(f,A,B)$, hence $\RC(f,A,B)$ is an image of $\Def(f,A,B)$ under $\piTR$. 
\end{proof}

\begin{corollary}
$\RC(f,A,B)$ is a smooth manifold.
\end{corollary}

\proof
$\RC(f,A,B)= \piTR(\Def(f, A, B))$. 

\noindent Since restriction of $\piTR$ to $\Def(f,A,B)$ is a Galois covering, $\RC(f,A,B)$ is a smooth manifold. A priori, it can
consist of several components.
\qed

\begin{lemma} We define
 \begin{equation*}\piRM\from\R(f) \to \Mod(B)\end{equation*}
 
 where $\piRM(f',A',B'):=B'$. The map $\piRM$ is a covering and
\begin{equation*}\pi=\piRM\circ\piTR. \end{equation*}

\end{lemma}




Let $M(f,A,B)=\pi_{RM}(\RC(f,A,B))$.

\begin{lemma} The map $\pi_{RM}:\RC(f,A,B)\to M(f,A,B)$ is a finite degree covering map.
\end{lemma}
\proof
We can also consider the space $\Rat(f, A, B)$ of all maps with the same critical portrait as $f$. We define $\Rat(f, A, B)$ as the space of pairs $g\from \chat \to \chat, h\from B\to \chat$ such that $g$ is holomorphic, $g \circ h = h\circ f$ on $A$, and the local degree of $f$ at $x$ equals to the local degree of $g$ at $h(x)$ for all $x \in A$.  We have a natural embedding of $\RC(f,A,B)$ into $\Rat(f,A,B)$. 
Hence, the restriction of $\piRM$ to $\RC(f,A,B)$ is necessarily a finite degree covering map. 
\qed

If it is one-to-one, we say that the deformation space is {\it projectable}.

We summarize the notations in the following diagram:

\begin{tikzcd}[column sep=tiny]
\Teich(B) \arrow[bend right=50]{dd}[swap]{\pi} \arrow{d}{\piTR} &\supset & \Def(f,A,B) \arrow[bend left=70]{dd}{\pi} \arrow{d}{\piTR}\\
\R(f)\arrow{d}{\piRM}& \supset & \RC(f,A,B)\arrow{d}{\piRM}\\
\Mod(B) & \supset & M(f,A,B)\\
\end{tikzcd}




\subsection{Quadratic rational map} \label{sec:quad_rat}
Our main example for this article is the space of quadratic rational maps.

Let $P(z),$ $Q(z)$ be polynomials  with no common roots. The map $f(z)=\frac{P(z)}{Q(z)}$ is called a degree $d:=\max\{\deg P, \deg Q\}$ rational map.   The space of rational maps of degree $d$ is denoted by $\Rat_d.$ Specifying dynamical relations for orbits of  critical points, one obtains subvarieties of $\Rat_d$ with prescribed dynamical properties.  M\"{o}bius transformations act on the space of rational maps by conjugating. The space of conjugacy classes of rational maps of degree $2$ with marked critical points is denoted by $\M2$. J. Milnor \cite{Milnor_quad_rational} showed that $\M2$ is biholomorphic to ${\mathbb C}^2$.

Quadratic rational maps have two critical points. We denote them by $c_1$ and $c_2$. {We will use the notation $v=f(c_2)$.} Consider the variety $\Per_n\subset \M2$:
$$\Per_n:=\{f\in \M2:\, f^{  n}(c_1)=c_1 \text{ and } f^{k}(c_1)\neq c_1 \text{ when } 0 < k < n\}.$$
$\Per_n$ is an algebraic subvariety of $\M2$. It is not known whether $\Per_n$ is irreducible for all $n$.

Suppose that $c_1, f(c_1),$ and $f^2(c_1)$ are mutually distinct for a map $f\in \M2$. Conjugating by a  M\"{o}bius transformation, we may assume that $c_1=0$ and $0\mapsto \infty \mapsto 1$. With this parametrization   $f$ takes the form 
$$z\mapsto 1+\frac{b}{z}+\frac{c}{z^2}=:f_{b,c}(z).$$ 
On the other hand, any two values of $b\in \mathbb{C}$, $c\in \mathbb{C}^*$ define a rational map in $\M2$.

Let
$$\Per'_n=\{f\in \Per_n: \, v\neq f^k(c_1), k=0,1,\dots\}$$
Let $f\in \Per'_n$. Take $A=\{c_1, f(c_1) \dots, f^{ (n-1)}(c_1)\}$, $B=A\cup \{v\}$. 
Then $\Def(f, A, B)$ is $1$-dimensional. 



\begin{lemma} Let $f\in \Per'_3$. The deformation space $\Def(f,A,B)$ is not projectable. The map $\piRM|_{\RC(f)}$ is a two-to-one covering. 
\end{lemma}

\begin{proof} Recall that with our standard parametrization, a map $f\in \M2$ has the form
$$f_{b,c}(z)= 1+\frac{b}{z}+\frac{c}{z^2}.$$
If $f_{b,c}\in \Per_3$, $f(1)=1+b+c=0$, so $c=1-b$.
We compute the second critical value $v=1-\frac{b^2}{4(1-b)}$. The last equation has two solutions $b_1,b_2$ for a fixed $v$, so the maps $f_{b_1, 1-b_1}$ and $f_{b_2, 1-b_2}$ project to the same point in the moduli space.
\end{proof}

\begin{lemma} Let $f\in \Per'_n$. For $n\geq 4$, the deformation space $\Def(f,A,B)$ is projectable.
\end{lemma}

\begin{proof} Let $\rho=f(1)$, $s=f(\rho)$. Then we have the following system of linear equations:
$$1+b+c=\rho$$
$$\rho^2+b\rho+c=\rho^2 s$$

If $n\geq 4$, $s\neq \infty$, so the system of linear equations is well-defined. Its determinant is equal to $(1-\rho)\neq 0$. Therefore, the values of $\rho$ and $s$ uniquely determine the map.
\end{proof}

\section{Equalizing multicurves} \label{sec:multicurves}

\subsection{Invariant multicurves for postcritically finite maps}

Let $f$ be a postcritically finite map with postctritical set $P_f$. The map $f$ induces the pull-back map on the space of simple closed curves on $\S^2\backslash P_f$. The pull-back is a composition of the pull-back under the covering map 
$f:(\S^2, f^{-1}(P_f))\to (\S^2, P_f)$ and forgetful map $i:(\S^2, f^{-1}(P_f))\to (\S^2, P_f)$. A simple-closed curve is called non-peripheral if it is not homotopically trivial and is not homotopic to a point in $P_f$. A multicurve $\Gamma=(\gamma_1,\dots,\gamma_n)$ is a collection of disjoint simple closed non-peripheral curves $\gamma_i$.   

We say that a multicurve $\Gamma=(\gamma_1,\dots, \gamma_n)$ is invariant if for every curve $\gamma_i$ all its preimages are either peripheral or homotopic to one of $\gamma_j$. We denote the preimages of $\gamma_i$ homotopic to $\gamma_j$ by $\gamma_{i j \alpha}$, where $\alpha$ is an enumerating index. We denote by $d_{i j \alpha}=\deg f\, :\, \gamma_i\to \gamma_{i  j  \alpha}$. For each multicurve $\Gamma$ there is a Thurston's matrix $T_{\Gamma}:$
$$(T_{\Gamma})_{i, j}=\sum_{\alpha} \frac{1}{d_{i j \alpha}}.$$ 
A multicurve is called {\it an obstruction} if the largest multiplier $\lambda(T_{\Gamma})\geq 1$.

\begin{theorem} (W. Thurston \cite{DH_Thurston}) The map $f$ with the hyperbolic orbifold is equivalent to a rational function if and only if there are no obstructions. 
\end{theorem}

\subsection{Equalizing multicurves for partially postcritically finite maps}
 
In this subsection we introduce analog of invariant multicurves for partially postcriticaly finite maps.

Consider a branched cover $f:(\S^2, A)\to (\S^2, B)$ such that $$B=A\cup f(A)\cup V(f).$$ 

The forgetful map  induces the map $i_*: \mathcal{C}(B) \to \mathcal{C}(A)$ on classes of simple closed curves:
$$i_*:[\gamma]_B \mapsto [\gamma]_A.$$
Similarly, we define the action of $i_*:\mathcal{M}(B) \to \mathcal{M}(A)$.

A formal linear sum 
$$m_1\gamma_1+\dots+m_n\gamma_n,$$
where $m_i\in {\mathbb Q}$ is called a weighted multicurve. Let $\mathcal{W}(P)$ be the space of equivalence classes of weighted multicurves on $({\mathbb S}^2,P)$. 
The forgetful map  extends linearly to weighted multicurves, hence, it defines a linear map $i_*\colon \mathcal{W}(B)\to \mathcal{W}(A)$.

Let $\beta_1,\dots,\beta_k\subset ({\mathbb S}^2,f^{-1}(B))$ be preimages of a simple closed curve $\gamma$ under the map $f$; denote by $d_i$ the degree of $f$ restricted to $\beta_i$. Define

$$\tilde{f}^*(\gamma) =\sum^k_{i=1}\frac{1}{d_{i}}\beta_i.$$

We extend this map linearly to $\tilde{f}^*: \mathcal{W}(B)\to \mathcal{W}(f^{-1}(B))$. 
Let
$$f^*:=\tilde{i}_* \circ \tilde{f}^*,$$ 
where $\tilde{i}_*$ is the forgetful map $\tilde{i}_*:\mathcal{W}(f^{-1}(B))\to \mathcal{W}(A)$. 

We also get a well-defined map
$$f^{*}:\mathcal{M}(P)\to \mathcal{M}(P),$$

if we ignore the weights in the definition of $f^*$.

To a pair of multicurves $\Gamma=\{\gamma_i\}$ on $({\mathbb S}^2, B)$ and $\Delta=\{\delta_j\}$ on $({\mathbb S}^2,A)$,  we associate the Thurston matrix $T_{\Gamma}^{\Delta}$ which 
is obtained from the map $f^*: \mathcal{W}(B)\to \mathcal{W}(A)$ by restricting the domain to $\mathbb{Q}^\Gamma$ and projecting the result to $\mathbb{Q}^\Delta$. 
The $(i, j)$-th entry of $T_{\Gamma}^{\Delta}$ equals to $\sum \frac{1}{d_{i,j,\alpha}}$, where $d_{i,j,\alpha}$ is the degree of $f$ restricted to $\delta_{j,\alpha}$ and $\{\delta_{j,\alpha}\}$ is the set of connected components of   $f^{-1}(\gamma_i)$ that are homotopic to $\delta_j$. In a similar fashion we define the restriction of $i_*$, the matrix $I_{\Gamma}^{\Delta}$, where  $(i, j)$-th entry is $1$ if $[\gamma_i]_A=[\delta_j]_A$ and $0$ otherwise.
Note that the matrices depend only on the equivalence classes of multicurves $\Gamma$ and $\Delta$.

We say that a multicurve $\Gamma$ is {\it equalizing}  if $[f^*\Gamma]=[i_*\Gamma]$ as multicurves.
A multicurve $\Gamma=\{\gamma_1,\ldots, \gamma_n\}$ is {\it arithmetically equalizing} if there exists a set of positive  rational weights 
$\{m_1,\ldots, m_n\}$ such that $i_*[m_1\gamma_1\dots+m_n\gamma_n]_B=f^*[m_1\gamma_1+\dots+m_n\gamma_n]_B$. 


\begin{theorem}\label{te:accumulation} $f^*\mathcal{S}_{\Gamma}=i_*\mathcal{S}_{\Gamma}$ if and only if $\Gamma$ is an equalizing multicurve. If the deformation space accumulates on $\mathcal{S}_{\Gamma}$, then  there exists $\tilde{\Gamma}\subset \Gamma$ such that $\tilde{\Gamma}$ is an arithmetically equalizing multicurve.
\end{theorem}

\begin{proof}  By Theorem \ref{te:aug_teich}, the pullback map  $f^*$ extends continuously to the augmented Teichm\"{u}ller space and $f^*(S_{\Gamma})=S_{f^*(\Gamma)}$. The forgetful map  extends continuously to the augmented Teichm\"{u}ller space as well. Hence, $i_*(S_{\Gamma})=S_{i_*(\Gamma)}$ and the first statement follows.

Let $\tau^k\in \Def(f,A,B)$ be a sequence of conformal structures such that $\tau^k\to \tau^*$, $\tau^*\in \mathcal{S}_{\Gamma}$. Let $m^k_i$ be the modulus of the maximal embedded annulus homotopic to $\gamma_i$ in the Riemann sphere corresponding to $\tau^k$. 
Set $\Delta= f^{-1}(\Gamma)$; Let $n^k_j$ be the modulus of the maximal embedded annulus homotopic to $\delta_j$ in the Riemann sphere corresponding to $f^*(\tau^k)$. Applying Theorem~7.1 from \cite{DH_Thurston}, we see that $n^k=T_{\Gamma}^{\Delta}m^k +O(1)$ where the additive constant depends only on $f$ and the supremum of conformal moduli of curves not in $\Gamma$. Similarly, we get $n^k=I_{\Gamma}^{\Delta}m^k +O(1)$ because $f^*(\tau)=i_*(\tau)$ by definition of  $\Def(f,A,B)$. Hence, $(T_{\Gamma}^{\Delta}-I_{\Gamma}^{\Delta})m^k=O(1)$. Since all entries of  $m^k$ tend to infinity, we see that there is a non-negative rational solution to $(T_{\Gamma}^{\Delta}-I_{\Gamma}^{\Delta})m=0$.
\end{proof}
This raises the question of which arithmetically equalizing curves give rise to accumulation points of a given deformation spaces. In Section \ref{sec:vanishing} we give the answer for arithmetically equalizing curves such that $[f^*\Gamma]=[i_*\Gamma]=\emptyset$.

\subsection{Elements of the mapping class group associated to equalizing multicurves.}

For a simple curve $\gamma$, we denote by $T_{\gamma}$ a Dehn twist along $\gamma$. 
\begin{lemma} Suppose that $m$ is such that, for every curve $\beta$ that maps to $\gamma$ by $f$, the degree of $f$ on $\beta$ divides $m$. Then 
$f^*(m \gamma) = \sum n_i \beta_i$, where the $n_i$ are integers, and $f^*(T_\gamma^m)=  \prod T_{\beta_i}^{n_i}$.  
\end{lemma}

\begin{proof} The map $T_{\gamma}$ is identity outside of an annulus $A_{\gamma}$ around the curve $\gamma$. Hence, we only need to check that it lifts in $A_{\gamma}$. The fact that it lifts in $A_{\gamma}$ is guaranted by the
condition on $m$
\end{proof}

\begin{note} Note that such $n$ always exist. For example, the map $T_\gamma^{d!}$ is always liftable, for every $\gamma$, where $d$ is the degree of $f$. 
\end{note}

Let $\Gamma=\{\gamma_1, \dots, \gamma_n\}$ be a multicurve. Let $\Delta=f^*(\Gamma)$


\begin{lemma}\label{lem:Dehn_twist} Let $\Gamma=\{\gamma_i\}$ be an arithmetically equalizing multicurve. Assume $f^*(m_1\gamma_1+\dots+m_n\gamma_n)=i_*(m_1\gamma_1+\dots m_n\gamma_n)$, where $m_i\in \Z$. Assume that $T^{m_k}_{\gamma_k}$ are liftable. Then $T^{m_1}_{\gamma_1}\circ \dots \circ T^{m_n}_{\gamma_n}\in G_{f}$. 
\end{lemma}

\begin{proof} Let $i_*(m_1\gamma_1+\dots m_n\gamma_n)=l_1\delta_1+\dots + l_k\delta_k$, where $\delta_i$ are different simple non-peripheral curves in $({\mathbb S}^2,A)$. 
Then $i_*\left(T^{m_1}_{\gamma_1} \circ \dots \circ T^{m_n}_{\gamma_n}\right)=f^*\left(T^{m_1}_{\gamma_1} \circ \dots \circ T^{m_n}_{\gamma_n}\right) = 
T^{l_1}_{\delta_1}\circ \dots\circ T^{l_k}_{\delta_k}$. 
\end{proof}

\begin{corollary} If $\Def(f,A,B)$  accumulates to a point $p\in \mathcal{S}_{\Gamma}$, $\Gamma=\{\gamma_1,\dots, \gamma_n\}$, then there exist $m_1,\dots, m_n\in \Z$ such that $T^{m_1}_{\gamma_n}
\circ \dots \circ T^{m_n}_{\gamma_n}\in G_f$.
\end{corollary}

\begin{proof} The corollary is a direct consequence of Theorem \ref{te:accumulation} and Lemma \ref{lem:Dehn_twist}
\end{proof}

\section{Period $4$ curve} \label{sec:per4}

Below we give a detailed description of $\Per_4\subset \M2$ (compare  \cite{Milnor_quad_rational}). 

Using our standard parametrization, the every map in $\Per_4$ takes the form 
$$f_{b,c}(z):=z\mapsto 1+\frac{b}{z}+\frac{c}{z^2}.$$ 

Let $\rho=f(1)$, then we have the following critical cycle: $0\to \infty\to 1\to \rho\to 0$. 

$$1+b+c=\rho$$ 
$$\rho^2+b\rho+c=0$$

$$
b=\frac{-\rho^2-\rho+1}{\rho-1}, \quad c=\frac{2\rho^2-\rho}{\rho-1}
$$

Hence, the elements of $\Per_4$ are parametrized by $\rho\in \hat{\mathbb C}\backslash\{0,1,\infty, \frac12\}$. We will denote by $f_{\rho}$ the unique map in $\Per_4$ with $f_{\rho}(1)=\rho$. The second critical point of $f_\rho$ is $$z_{\rho}=-\frac{2c}{b}=2\frac{2\rho^2-\rho}{-\rho^2-\rho-1}.$$ The corresponding critical value of $f_\rho$ is
$$v_{\rho}=1-\frac{b^2}{4c}.$$ 

$$v_{\rho}=1-\frac{(\rho^2+\rho-1)^2}{4(\rho-1)(2\rho^2-\rho)}.$$

Fix a map $f_{\rho}\subset \Per_4$ such that $v_{\rho}\neq 0,1,\rho,\infty$. 
Let                                                         
$\Per'_4\subset \Per_4$,
$$\Per'_4:=\{f_{\rho}\in \Per_4:\, v_{\rho}\neq 0,1,\rho\}.$$
$\Per_4\setminus \Per'_4$ {consists of 6 points}: 2 solutions of the equation $v_{\rho}=0$ ($\rho_0^1=\frac{3-\sqrt{5}}{2}$, $\rho_0^2=\frac{3+\sqrt{5}}{2}$), 2 solutions of the equation $v_{\rho}=1$ ($\rho_1^1=\frac{-1-\sqrt{5}}{2}$, $\rho_1^2=\frac{-1+\sqrt{5}}{2}$), 2 solutions of the equation $v_{\rho}=\rho$ ($\rho_{\rho}^1=\frac{3-i\sqrt{3}}{6}$, $\rho_{\rho}^2=\frac{3+i\sqrt{3}}{6}$).

The map $I\from\Per'_4\to \Mod(B)$
$$I(f_{\rho})=(0,\infty,1,\rho,v_{\rho})$$ 
is an embedding, thus, $\Per'_4$ is a submanifold of $\Mod(B)$.

In this section we study the deformation space  $\Def(f_{\rho},A,B)$, $A=(0,1,\infty,\rho)$, $B=(0,1,\infty,\rho, v_{\rho})$ for an arbitrary $f_\rho \in \Per'_4$.

\subsection{Examples of  of Equalizing multicurves}. 

\begin{enumerate}
\item  Consider a curve $\delta$ that separates $v$  and $\rho$  from the rest of the points. All such curves are arithmetically equalizing, since $f^*(\delta)=i_*\delta=\emptyset$ (See figure \ref{fig:1}) 

\begin{figure}[h!]
\psfrag{0}{$0$}
\psfrag{r}{$\rho$}
\psfrag{i}{$\infty$}\psfrag{1}{$1$}\psfrag{v}{$v$}
\psfrag{g}{$\delta$}
\includegraphics[height=4.5cm]{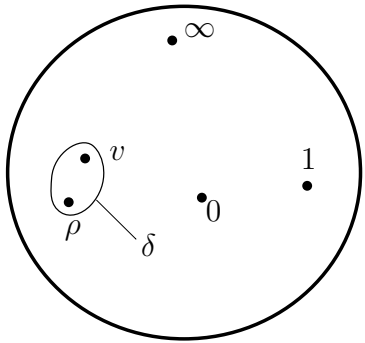}
\caption{Equalizing curve $\delta$}\label{fig:1}
\end{figure}

\item The curve $\gamma$ separating $0,1$ from $v, \rho, \infty$: 
 \begin{lemma} For $\rho\in {\mathbb R}_+$, let $\gamma_1=\{\sqrt{\rho}e^{i\theta}\colon \theta\in [0,2\pi] \}$. Then, for big enough $\rho$, $\gamma=f_{\rho}(\gamma_1)$ is an arithmetically equalizing 
curve  and it separates $0,1$ from $v,\rho,\infty$. 
\label{l:AEC}
\end{lemma}

\begin{proof} 
Recall that
$$f_{\rho}(z)=1+\frac{b}{z}+\frac{c}{z^2},$$ 
$b=-\rho+O(1)$, $c=2\rho+O(1)$.
Therefore 
$f_{\rho}(\sqrt{\rho} e^{i\theta})=-\sqrt{\rho}e^{-i\theta}+O(1)$. 
Since $v_{\rho}=-\frac{\rho}{8}+O(1)$ as $\rho\to \infty$,  $\gamma$ lies in a collar around $\beta$ that does not contain $0,1,\infty,\rho, v_{\rho}$. We see that $\gamma_1$ and $\gamma$ are homotopic relative to $B$.

Let   $\gamma_2$ be the second connected component of the preimage of $\gamma$; then $\gamma_2$ bounds unmarked preimages of $0$ and $1$, so it is homotopic to a point in 
$\hat{\mathbb C}\backslash B$. Hence $\gamma$ is an arithmetically equalizing curve.
\end{proof}

\end{enumerate}







Let $\delta$ be a simple closed curve that encircles $v$ and $\rho$ and does not intersect $\gamma$. Even though $\delta$ is an arithmetically equalizing curve, the deformation space $\Def(f,A,B)$ does not accumulate on $\mathcal{S}_{\delta}$.

\begin{lemma} \label{lem:defnotacc} The deformation space $\Def(f,A,B)$ does not accumulate on $\mathcal{S}_{\delta}$.

\end{lemma} 
\begin{proof}
Note that $\delta$ splits the sphere into two components, one with three marked points and a node and one with two marked points and a node. Since there is a unique conformal structure on the latter sphere $\mathcal{S}_\delta \cong \Teich(A)$. On the other hand, since $\delta_1$ is peripheral $\mathcal{S}_{\delta_1}=\mathcal{S}_\emptyset=\Teich(A)$. We see that $i_*$ restricts to ${\id}_{\Teich(A)}$ on $\mathcal{S}_\delta$.

Similarly, we see that the action of $f^*$ on $\mathcal{S}_\delta$ restricts to the Thurston pull-back map $g^*\colon \Teich(A) \to \Teich(A)$ where $g$ is a postcritically finite branched cover obtained from $f$ by pinching $\delta_1$ in the domain and $\delta$ in the range. The map $g$ is obstructed: $\gamma$ descends to a Levy cycle for $g$.

By continuity, any accumulation point $\tau$ of  $\Def(f,A,B)$ satisfies $f^*(\tau)=i_*(\tau)$. Thus, any such accumulation on $\mathcal{S}_\delta$ would be a fixed point of $g^*$. Thurston's Theorem \cite{DH_Thurston} guarantees that there are no fixed points for obstructed branched covers.
\end{proof}

\section{Vanishing arithmetically equalizing curves and postcritically finite maps} \label{sec:vanishing}
Let $n$ be the dimension of the deformation space $\Def(f,A,B)$. Let $\Gamma=\{\gamma_1, \dots, \gamma_n\}$ be a mulicurve on $(\S^2, B)$ consisting of exactly $n$ curves. We denote by $f^*_{\Gamma}$ the extension of $f^*$ to the stratum $S_{\Gamma}$.
Assume that $f^*\Gamma=i_*\Gamma=\emptyset$.

\begin{lemma}   The  surface obtained by pinching the multicurve $\Gamma$ on $(\S^2, B)$ has only one component with more than three special points.
\end{lemma}

\begin{proof}

Since $i_*\Gamma=\emptyset$, every curve $\gamma_i$ in $\Gamma$ has at most one point of $A$ in one of its complementary components $U_i$; for $\gamma_i$ to be non-peripheral in $\S^2, B$, the component $U_i$ must contain at least one point of $B \setminus A$. If all $U_i$ are mutually disjoint, then each of them contains exactly one point of $B \setminus A$ because   $\# B \setminus A=n$. If any of $U_i$ are nested, the annulus between adjacent $\gamma_i$ will contain exactly one point of  $B \setminus A$.
\end{proof}

We denote the component with $3$ punctures from the previous lemma by $R_{\Gamma}$. 

\begin{corollary} The map $f:(\S^2, B)\to (\S^2, B)$ induces a postcritically finite map $f_{\Gamma}:R_{\Gamma}\to R_{\Gamma}$.  
\end{corollary}
Note that the image of $i_{*}(S_{\Gamma})$ is in $\Teich(\S^2,A)$ and the map is one-to-one. Hence the map $i_{*}^{-1}\cdot f^*:S_{\Gamma}\to S_{\Gamma}$ is well-defined and is the Thurston pullback map for the map $f_{\Gamma}$.

\begin{theorem} \label{thm:pcf} Assume $\Gamma$ is an equalizing multicurve for the deformation space $\Def(f,A,B)$ such that $f^*\Gamma=i_*\Gamma=\emptyset$. Assume that the postcritically finite map $f_{\Gamma}$ has a hyperbolic orbifold. Then the deformation space $\Def(f, A,B)$ accumulates on the stratum $S_{\Gamma}$ if and only if $f_{\Gamma}$ is Thurston equivalent to a rational map. \label{pcf}
\end{theorem} 

\begin{proof}
If the deformation space accumulates on a point $\tau$ in the stratum $S_\gamma$, that point must satisfy $f^*(\tau)=i_*(\tau)$. That makes $\tau$ the fixed point of  the Thurston pullback map for the map $f_{\Gamma}$, existence of which proves that $f_{\Gamma}$ is Thurston equivalent to a rational map.

On the other hand, let us assume that $f_\Gamma$ is Thurston equivalent to a rational map. Then there is a unique point $\tau$ on $S_\Gamma$ such that $f^*(\tau)=i_*(\tau)$. Following J.Hubbard and S.Koch \cite{HubbardK}, let us consider the space $$\qg =\bigcup_{\Gamma' \subset \Gamma} S_{\Gamma'} / \Delta_\Gamma,$$ where $\Delta_\Gamma$ is the subgroup of the $\text{MCG}(B)$ which is generated by Dehn twists about the curves of $\Gamma$. By Theorem 10.1 in \cite{HubbardK}, the set $\qg$ is a complex manifold. 

Note that since $f^*\Gamma=\emptyset$, the action of $f^*$ is equivariant with respect to $\Delta_\Gamma$: for any $\phi \in  \Delta_\Gamma$ we have $f^*=(\phi\circ f)^*$ on $\Teich(B)$. The same 
equivariance property similarly holds for $i_*$. Therefore $f^*$ and $i_*$ project to analytic maps from $\qg$ to $\Teich(A)$. Note that by assumption, $f_\Gamma$ has hyperbolic orbifold, so $\tau$ is 
the only point in $\S_\Gamma$ such that $i_* \tau=f^*\tau$. Since on the stratum $S_{\gamma}$, $f_{\gamma}$ is a hyperbolic postcritically finite map, by the proof of the Thurston's theorem on the stratum $S_{\Gamma}$, 
 $(f^*-i_*)$ is invertible. 
Hence, by the implicit function theorem,  the projection of $\Def(f,A,B)$ to $\qg$ together with $\tau$ is an $n$-dimensional manifold in a neighborhood of $\tau$.  In particular, 
 $\tau$ is an accumulation point of $\Def(f,A,B)$.
\end{proof}

Theorem \ref{pcf} allows to prove that certain deformation spaces are non-empty.
\begin{lemma} \label{lem:one_cycle} Let $f$ be a branched cover with exactly two simple critical points $a$ and $b$. If $a$ is periodic and $b = f(a)$, then $f$ is equivalent to a quadratic rational map.
\end{lemma}
\begin{proof}
This follows directly from Theorem 4.1 \cite{Levy} which states that if a degree 2 branched cover has a Thurston obstruction, it must have a Levy cycle.  Suppose $\Gamma=\{\gamma_i\}$ is such a Levy cycle; changing $f$ up to homotopy, we may assume that $\Gamma$ is forward invariant. Each $\gamma_i$ has a preimage component of degree 1, therefore $\gamma_i$ does not separate $a$ and $b$. The connected component of $\S^2 \setminus \Gamma$ containing $a$ and $b$ is then forward invariant (recall that $f(a)=b$) which implies that it contains full orbit of $a$. This contradicts the fact that $\gamma_i$ were assumed to be non-trivial.
\end{proof}


\begin{corollary} Let $f$ be a branched cover with exactly two simple critical points. Assume that the orbit of the first critical value $u$ is periodic and the second critical value $v$ does not belong to the forward orbit of $u$. Let $A$ denote the forward orbit of $u$ and $B=A\cup\{v\}$. Then $\Def(f,A,B)$ is not empty.
\end{corollary}

\begin{proof}

Take the point $w=f(u)$ of $A$. Consider a curve $\gamma$ that separates $w$ and $v$ from the rest of points in $B$. Then $\Gamma=\{\gamma\}$ satisfies the first assumption of Theorem~\ref{thm:pcf}; the map $f_\Gamma$ is a degree 2 postcritically finite branched cover for which $w$ becomes a critical value. Both critical values of $f_\Gamma$ are periodic, hence $f_\Gamma$ is equivalent to a rational map by the previous lemma. Note that we can assume that $A$ has at least 4 points, otherwise the statement of this corollary is trivially true. Therefore, the map $f_\Gamma$ has hyperbolic orbifold and all assumptions of Theorem~\ref{thm:pcf} are satisfied forcing $\Def(f,A,B)$  to accumulate on the stratum $S_\Gamma$ and, in particular, not be empty.
\end{proof}

%
%
%
In Lemma~\ref{lem:defnotacc} we  gave an example of a curve that is arithmetically equalizing, but the deformation space does not accumulate to the corresponding stratum.







\section{Topology of Deformation Space, non-contractibility}\label{sec:noncontractible}

In this section for we assume that $\Def(f,A,B)$ is projectable.

J.Milnor \cite{Milnor_quad_rational} asked whether $\Per'_n$ is irreducible. In our terminology it is equivalent to asking whether the space $M(f, A,B)$ contains one component. A priori, it can contain several components. Let $M_0$ be one of the components. Denote by $j\from M_0\to \Mod(B)$ the inclusion map. Fix a point $m_0\in M_0$. The inclusion map induces a homomorphism $j_*\from \pi_1(M_0, m_0)\to \pi_1(\Mod(B), j(m_0)).$  The group  $\pi_1(\Mod(B), j(m_0))$ can be identified with $\Map(B)$. 

Let $G_0:=I_*(\pi_1(M_0,m_0)) \cong \pi_1(M_0,m_0)/\ker j_*$ Let $\Def_0$ be a component of the deformation space $\Def(f,A,B)$ such that $\pi(\Def_0)=M_0$. 

\begin{lemma} \label{lem:fund_group_comp} The group $G_0$ is a subgroup of $G_f$. The group $G_0$ preserves $\Def_0$, $\pi_1(\Def_0)=\ker j_*$. Hence, $\Def_0$ is contractible if and only if $\ker j_*=\emptyset$.
\end{lemma}

\begin{lemma} The deformation space $\Def(f, A,B)$ is contractible if and only if $\ker j_*=\emptyset$ and $G_0=G_f$.
\end{lemma}

In particular, we have the following lemma that guarantees that certain deformation spaces are not contractible.

\begin{lemma} \label{lem:nc} Let $\Def(f,A,B)$ be a $1$-dimensional deformation space. Assume that $\gamma_1$ and $\gamma_2$ are disjoint   arithmetically equalizing curves. Then $\Def(f,A,B)$ is not contractible. \end{lemma}

\begin{proof} Assume the contrary. Each component $M_0$ of $M(f,A,B)$ is a  Riemann surface with at least one puncture. Hence, $\pi_1(M_0)$ is a free group. If deformation space is contractible,  then $\pi_1(M_0)=G_f$. There exists $m_1,$ $m_2$ such that $T^{m_1}_{\gamma_1}, T^{m_2}_{\gamma_2}\in G_f$. This is a contradiction since $T^{m_1}_{\gamma_1}$ commutes with $T^{m_2}_{\gamma_2}$.  
\end{proof}

\subsection{Non-contractibility of the deformation space for maps in $\Per'_4$}

\begin{lemma} The deformation space $\Def(f,A,B)$, where $f\in \Per'_4$ is not contractible.
\end{lemma}

\begin{proof}

Let $\delta$ be a curve that bounds $\rho$ and $v$, and does not intersect $\gamma$.

\begin{figure}[h!]
\psfrag{0}{$0$}
\psfrag{r}{$\rho$}
\psfrag{r'}{$\rho'$}\psfrag{i}{$\infty$}\psfrag{i'}{$\infty'$}\psfrag{d2}{$\delta_2$}\psfrag{1}{$1$}\psfrag{v}{$v$}
\psfrag{1'}{$1'$}\psfrag{g1}{$\gamma_2$}\psfrag{g2}{$\gamma_1$}\psfrag{d1}{$\delta_1$}\psfrag{c}{$c$}\psfrag{d}{$\delta$}\psfrag{g}{$\gamma$}

\includegraphics[height=4.5cm]{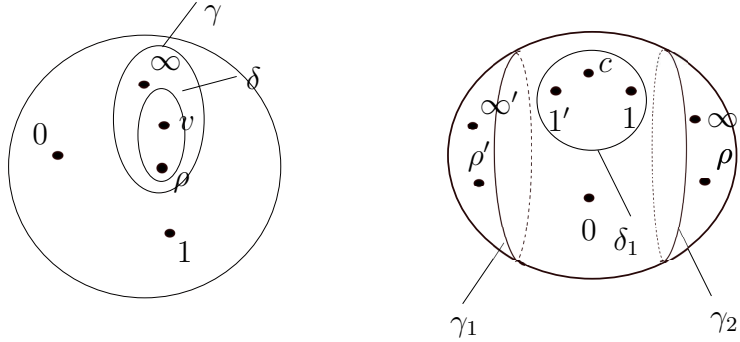}
\caption{A pair of equalizing curves $\gamma$ and $\delta$. Here and later: by  $x'$ we denote a point that has the same image as $x$. }\label{fig:gd}
\end{figure}

We see that $[\delta]_A$ is peripheral as well as $[\delta_1]_A$ where $\delta_1=f^{-1}(\delta)$ is the only connected component of the preimage of $\delta$. We conclude that $\delta$ is arithmetically equalizing. 
 
By Lemma~\ref{lem:nc}  $\Def(f, A,B)$ is non-contractible. 
\end{proof}

\subsection{Non-contractibility of the deformation spaces for maps in $\Per'_{n}$}

In this section we give an explicit construction of a pair of non-intersecting equalizing curves for maps in $\Per'_{n}$. 
This yields the following result.
\begin{theorem} There exists $f\in \Per'_n, n\ge 4$ such that $\Def(f,A,B)$  is not contractible.
\end{theorem}

Recall that with the parametrization where  $0$ is a critical point, $0\to \infty\to 1$, a quadratic rational map takes the form
$$f_{b,c}(z):=1+\frac{b}{z}+\frac{c}{z^2},$$ 

We denote by $\rho_k(b,c):=f_{b,c}^{k-2}(1)$, $\rho_k=\frac{P_k(b,c)}{Q_k(b,c)}$, where $P_k(b,c)$ and $Q_k(b,c)$ are polynomials defined recursively by
$$P_2(b,c)=Q_2(b,c)=1$$
$$ P_{k+1}(b,c)= P_{k}^2(b,c)+bP_{k}(b,c)Q_{k}(b,c)+cQ_{k}^2(b,c),$$ $$ Q_{k+1}(b,c)=P_{k}^2(b,c).$$
We see that for odd $k$, $\deg P_k=\deg Q_k+1$ and for even $k$, $\deg P_k=\deg Q_k$. 

The map $f_{b,c}\in \Per_n$  if and only if $\rho_{n}(b,c)=0$. 
 
Let $p_k$, $q_k$ be the highest homogeneous part of $P_k$ and $Q_k$ correspondingly. Then for odd $k$,
$$p_{k+1}=p^2_{k}+b p_k q_k, \quad q_{k+1}=p_k^2.$$
For even $k$, 
$$p_{k+1}=q_k(b p_k+c q_k),\quad q_{k+1}=p_k^2$$ 

We will denote by $\tp_k$, $\tq_k$ reduced highest homogeneous parts of $P_k$ and $Q_k$, obtained by canceling the common factors of $p_k$ and $q_k$. Reduced homogenous parts satisfy the relation:
 $$\tp_{k+1}=\tp_{k}+b \tq_k, \quad \tq_{k+1}=\tp_k \quad \mbox{for odd k;}$$   
 $$\tp_{k+1}=\tq_k(b \tp_k+c \tq_k), \quad \tq_{k+1}=\tp_k^2 \quad \mbox{for even k.}$$
 
 Since $\tp_{2}=1$, $\tq_{2}=1$ do not have common factors, we can show inductively that $\tp_k$ and $\tq_k,$ for $k\geq 2$ do not in fact have common factors. 
 
 The equation $\rho_{n}=0$ defines an algebraic curve.  The intersection of this curve with infinity in $\mathbb{CP}^2$ is given by solutions of the homogenous equation $\tp_n(b,c)=0$.
 
  
\begin{lemma} \label{lem:eq_curve} Consider $f_{b,c}\in \Per'_n$ such that $b,c$ are large, and $\frac{b}{c}$ is close to $\alpha$.Let $\gamma_1=\{\sqrt{|c|}e^{i\theta},\theta\in [0,2\pi] \}$.  Then $\gamma=f_{b,c}(\gamma_1)$ separates $0,1, \rho_4,\dots, \rho_{n-2}$ from $v,\infty, \rho_3\dots, \rho_{n-1}$. Moreover, it is arithmetically equalizing. 
\end{lemma}

\begin{proof} 

Note that if $\tp_k(b,c)=0$ then $\tp_{km}(b,c)=0$ for all $m$. Also, if  $\tp_{k}(b,c)=0$, then  $\tp_{k+2m}(b,c)=0$ and $\tq_{k+1+2m}(b,c)=0$ for all $m$ if $k$ is odd.

We see that if $n$ is odd, $\tp_n=\tp_{n-2} (b^2 \tq_{n-2}+(b+c)\tp_{n-2})$ has a solution $(\alpha c,c)$ which is not a solution of $\tp_k$ for $k<n$. On the other hand if $n$ is even, then any solution  $(\alpha c,c)$ of $\tp_n$ is not a solution of $\tp_{2k+1}$ for $2k+1<n$.

We observe that under the assumptions $$f_{b,c}(\sqrt{|c|} e^{i\theta})=\alpha \sqrt{|c|}e^{-i(\theta+\arg c/2)}+O(1).$$

Since $\alpha$ is not a solution of $\tp_k(\alpha c,c)$, $\deg p_k(\alpha c, c)>\deg q_k(\alpha c,c)$ for odd $k$, $\deg p_k(\alpha c, c)\le \deg q_k(\alpha c,c)$ for even $k$. 
This yields  $\rho_k(\alpha c, c)=\frac{p_k(\alpha c, c)}{q_k(\alpha c, c)}\ge a_k c+O(1)$ ($a_k\neq 0$) for odd $k$, $\rho_k(\alpha c, c)=\frac{p_k(\alpha c, c)}{q_k(\alpha c, c)}\le b_k+O(1)$ ($b_k\neq 0$) for even $k$. Therefore,  for big enough $c$, $\gamma$ has the right separation property, and it lies in an annular neighborhood of $\gamma_1$ that does not contain any marked points.

Let $\gamma_2$ be the second preimage of the curve $\gamma$. For even $n$, $\gamma_2$ bounds unmarked preimages of $\infty, \rho_3,\dots, \rho_{n-1}$. For odd $n$, $\gamma_2$ bounds $\rho_{n-1}$ and unmarked preimages of $\infty, \rho_3,\dots, \rho_{n-2}$.
\end{proof}

\begin{figure}[h!]

\psfrag{i}{$\infty$}
\psfrag{0}{$0$}
\psfrag{1}{$1$}\psfrag{1'}{$1'$}\psfrag{i}{$\infty$}\psfrag{v}{$v$}\psfrag{i'}{$\infty'$}\psfrag{c}{$c$}
\psfrag{r1}{$\rho_3$}\psfrag{r2}{$\rho_4$}\psfrag{r3}{$\rho_5$}\psfrag{n}{$\rho_{n-1}$}\psfrag{n-1}{$\rho_{n-2}$}\psfrag{'n}{$\rho'_{n-1}$}\psfrag{'1}{$\rho'_3$}
\psfrag{g}{$\gamma$}\psfrag{d}{$\delta$}\psfrag{g1}{$\gamma_1$}\psfrag{d1}{$\delta_1$}\psfrag{g2}{$\gamma_2$}
\includegraphics[height=6cm]{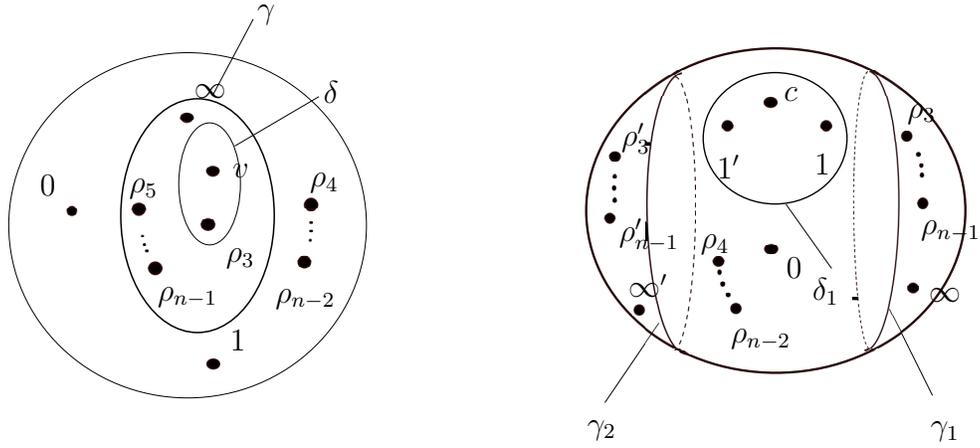}
\caption{A pair of equalizing curves for a map $f\in \Per_n$, $n$ even}\label{fig:even}
\end{figure}

\begin{figure}[h!]

\psfrag{i}{$\infty$}
\psfrag{0}{$0$}
\psfrag{1}{$1$}\psfrag{1'}{$1'$}\psfrag{i}{$\infty$}\psfrag{v}{$v$}\psfrag{i'}{$\infty'$}\psfrag{c}{$c$}
\psfrag{r1}{$\rho_3$}\psfrag{r2}{$\rho_4$}\psfrag{r3}{$\rho_5$}\psfrag{n}{$\rho_{n-1}$}\psfrag{n-1}{$\rho_{n-2}$}\psfrag{n-2}{$\rho_{n-3}$}\psfrag{'n}{$\rho'_{n-1}$}\psfrag{'1}{$\rho'_3$}
\psfrag{g}{$\gamma$}\psfrag{d}{$\delta$}\psfrag{g1}{$\gamma_1$}\psfrag{d1}{$\delta_1$}\psfrag{g2}{$\gamma_2$}
\includegraphics[height=6cm]{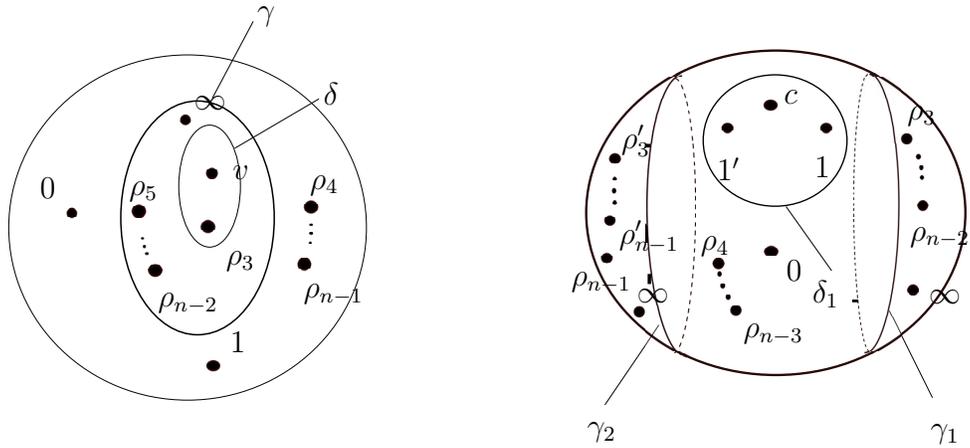}
\caption{A pair of equalizing curves for a map $f\in \Per_n$, $n$ odd}\label{fig:odd}
\end{figure}

\begin{lemma} For any $n\geq 3$ there exists $f\in \Per'_n$, and a pair of non-intersecting equalizing curves $\gamma$ and $\delta$.
\label{lem:pernc}
\end{lemma}

\begin{proof} Take $f=f_{b,c}$ and $\gamma$ from the previous lemma. Take a curve $\delta$ that encircles  $v$ and $\rho_1$ and does not intersect $\gamma$. \end{proof}

\subsection{Examples of self-equivalences that preserve a component of the deformation space}

Let $\Gamma=\{\gamma_1,\dots, \gamma_n\}$ be a multicurve. Let $m_0\in M(f,A,B)$, we can identify $\pi(\Mod(B), m_0)$ with $\Map(B)$. 

\begin{lemma} \label {lem:punctures} Consider a one-dimensional deformation space $\Def(f,A,B)$. Let $\alpha:[0,1)\to \Def(f,A,B)$ be a path, such that $\lim_{t\to 1}\gamma(t)\in \mathcal{S}_{\Gamma}$, $\alpha(0)=m_0$. Then there is a   map that associates to $\alpha$, an element $g_{\alpha}\in G_f$, such that $g_{\alpha}=T^{m_1}_{\gamma_1}\dots T^{m_n}_{\gamma_n}\,$ with $m_1,\dots, m_n\in \Z \setminus \{0\}$.   Moreover, $g_{\alpha}$ preserves the component of the deformation space that contain $\alpha$. This also defines a natural map that associates to $\alpha$, $\Gamma_{\alpha}=m_1 \gamma_1+\dots +m_n\gamma_n$ an arithmetically equalizing multicurve.  
\end{lemma}


\begin{proof}
We first observe the following. Suppose that $U$ is a holomorphic disk, $z \in U$, and $q\from U\setminus\{z\}\to \Mod(B)$ is a holomorphic map  that extends continuously (across $z$) to map $U$ into the Deligne-Mumford compactification of $\Mod(B)$, and such that $q(z) \in \mathcal{S}_\gamma$ for some curve $\gamma$. Now let $\beta$ be a curve in $U \setminus\{z\}$ that winds nontrivially around $z$; then $q(\beta)$ induces a nontrivial element of $\Mod(B)$, which in fact is  $T_\gamma^m$ for some $m \neq 0$. Similar arguments show  that if $q(z) \in \mathcal{S}_\Gamma$ for some multicurve $\Gamma = \{\gamma_i\}$, then $q(\beta)$ induces a map of the form $\prod T_{\gamma_i}^{m_i}$ where all $m_i$ are nonzero. 

Now in the setting of the Lemma, $\RC(f, A, B)$ is a Riemann surface of finite type, and we must have $\piTR(\alpha(t)) \goesto z$ as $t \goesto 1$, for some $z \in \overline{\RC(f, A, B)} \setminus \RC(f, A, B)$. Thus we can find a disk $U \subset \overline{\RC(f, A, B)}$ that maps into $\Mod(B)$ (by $\piRM$) as described in the previous paragraph. 


Then take any closed curve $\beta$ in $\Mod(B)$ that is the image of a curve in $U \setminus \{z\}$ that winds once around $z$. We find that the image $g_{\alpha}=q_{\beta}$ of $\beta$ in $\Map(B)$ has the form $g_{\alpha}=T^{m_1}_{\gamma_1}\dots T^{m_n}_{\gamma_n}$, where each $m_n \neq 0$. 
\end{proof}

\section{Appendix: Classification of equalizing multicurves for $\Per'_4$ curve}

We say that two multicurves $\Gamma_1$ and $\Gamma_2$ in $(\S^2, B)$ are in the same homology class if there exist an element $h\in \Map(B)$ such that $\Gamma_2=h \Gamma_1$. Note that if two multicurves are in the same homology class then $S_{\Gamma_1}$ and $S_{\Gamma_2}$ in the augmented Teichm\"{u}ller space $\overline{\Teich}(B)$ project to the same stratum in the Deligne-Mumford compactification of $\Mod(B)$. 
The homology class of a multicurve is determined by the partitions the curves of the multicurve induce on the set $B$.

If the set $B$ contains five points, each non-empty multicurve contains one or two curves. In this case, the homology class of a multicurve $\Gamma$ is determined by the partition of these five points into 1) two subsets if $\Gamma$ contains one curve, 2) three subsets if $\Gamma$ contains two curves. 

We consider a refined notion of equivalence that we call {\it $f$-equivalence}: two multicurves $\Gamma_1$ and $\Gamma_2$ are {\it $f$-equivalent} if there exists a liftable element $h$ of $\Map(B)$ such that $\Gamma_1=h\Gamma_2$. 

\begin{lemma} The map $f^*$ is well-defined as a map from classes of $f$-equivalent multicurves to homology classes of multicurves.
\end{lemma}

\proof
If $h$ is a liftable element and $h\Gamma_1=\Gamma_2$ then
$$(f^*h) (f^*\Gamma_1)=f^*(h \Gamma_1)=f^*(\Gamma_2).$$
Hence, $f^*\Gamma_1$ and $f^*\Gamma_2$ belong to the same homology class.
So, the map $f^*$ is well-defined on the classes of $f$-equivalence.
\qed

Note that there are $[H_f:\Map(B)]$ $f$-equivalence classes for each homology class of multicurves.

\begin{lemma} The map $i_*$ is well-defined on homology classes and on $f$-equivalence of  multicurves. 
\end{lemma}

\proof Let $h\in \Map(B)$ be such that $\Gamma_2=h\Gamma_1$. Then 
$$(i_*h) (i_*\Gamma_1)= i_*(h\Gamma_1)=i_*\Gamma_2.$$ So, $i_*$ is well-defined on homology classes. As a corollary it is well-defined on $f$-equivalence classes. 
\qed

\begin{lemma} Suppose that the deformation space $\Def(f,A,B)$ is one-dimensional. Then there is a natural map from punctures of $M(f,A,B)$ to $f$-equivalence classes that contain arithmetically equalizing multicurves.   
\end{lemma}

\begin{proof} Let $m_0$ be a base point on a one-dimensional deformation space $\Def(f,A,B)$. 
Consider a path $\alpha$ in the moduli space that starts at $\pi(m_0)$ and ends at  a puncture $p$ of  $M(f,A,B)$, i.e.  $\alpha\from [0,1)\to M(f,A,B)$,  $\alpha(0)= \pi(m_0)$, and  $\lim_{t\to1}\alpha(t)=p$.
By Lemma \ref{lem:punctures},  each such path when lifted to the deformation space $\Def(f,A,B)$ defines an equalizing multicurve $\Gamma_{\alpha}$ and an element of the mapping class group $g_{\alpha}$, which is a composition of Dehn twists about curves in $\Gamma_{\alpha}$. For any two such paths $\alpha$ and $\beta$, there exists $h\in j_*\pi_1(M(f,A,B))\subset \Map(B)$ such that $h g_{\alpha}h^{-1}=g_{\beta}$. Thus, $h(\Gamma_{\alpha})=\Gamma_{\beta}$. Since $h\in j_*\pi_1(M(f,A,B))$, $h$ is liftable. Hence, the map from punctures to $f$-equivalence classes of arithmetically equalizing multicurves is well-defined. 
\end{proof}

Now let us consider in more details what happens for $\Per'_4$. Recall that $\Per'_4$ is a sphere with $10$ punctures: $0$, $1$, $\infty$, $1/2$, $\rho^0_1$, $\rho^0_2$, $\rho^1_1$, $\rho^1_2$, $\rho^{\rho}_1$, $\rho^{\rho}_2$. If $\Gamma_1$ and $\Gamma_2$ are $f$-equivalent, then partitions of $A$ defined by $f^*\Gamma_1$ and $f^*\Gamma_2$ are the same. So $f$-equivalence classes are classified by pairs of partitions $(P_B, P_A)$ such that partition of $A$ can be induced by partition of $B$.   
 
We classify $f$-equivalence classes that contain equalizing multicurves. If we take a multicurve $\Gamma$, such that $f^*\Gamma$ and $i_*\Gamma$ induce the same partition of $A$, then there is an equalizing multicurve in the $f$-equivalence class of $\Gamma$. In the case of $\Per'_4$ this can be verified by explicitly constructing these curves, for example, in Lemma~\ref{l:AEC} we construct such a curve for case (6). We conjecture the same holds in the general case.

Note that not all multicurves in a $f$-equivalence class are equalizing (except for some special cases).

In the classification below we note which equivalence classes contain multicurves that are images of punctures in the sense of the previous lemma. 
In each case we list the partition of $B$ and draw the equalizing curves and their preimages.
We denote by $\infty', 1', \rho'$ unmarked preimages of $1, \rho, 0$.

\begin{table}[]
\begin{tabular}{|c|c|c|c|c|}
\hline
&Partition of $B$ & Arithmetically Equalizing & Puncture of $M$ & Thurston Matrix   \\ \hline
 1&  $ (v,0)(\infty, 1,\rho)$  & Yes &$\rho_1^0$, $\rho_2^{0}$ & 0 \\ \hline
2 &   $ (v,1)(0,\infty, \rho)$ & Yes & $\rho_1^{1}$, $\rho_2^{1}$ & 0\\ \hline
3 & $ (v, \rho)(0,\infty, 1)$   & Yes & $\rho_1^{\rho}$, $\rho_2^{\rho}$  &0\\ \hline
4 & $ (v, \infty)(0,1,\rho)$  & Yes & 1/2  &0\\ \hline
5 & $ (v,0,1)(\infty, \rho)$  & No &    &1/2\\ \hline
6 & $(v, \infty,\rho)(0,1)$ & Yes & $\infty$  &1\\ \hline
7 & $ (v,\rho)(\infty)(0,1)$   & Yes &    & $\begin{bmatrix}
1 & 0 
\end{bmatrix}$\\\hline
8 & $(v,1)(0)(\infty,\rho)$   & No &    & $\begin{bmatrix}
1/2 & 0 
\end{bmatrix}$\\ \hline
9 & $(v,0)(1)(\infty,\rho)$   & No &    & $\begin{bmatrix}
1/2 & 0 
\end{bmatrix}$\\ \hline
10 & $(v,\infty)(0)(1,\rho)$   & Yes &  1 & $\begin{bmatrix}
1 & 0 
\end{bmatrix}$\\ \hline
11 & $(v,\infty)(1)(0,\rho)$   &  Yes & 0 & $\begin{bmatrix}
1 & 0 
\end{bmatrix}$\\ \hline
12a & $(v,\infty)(\rho)(0,1)$   & No &   & $\begin{bmatrix}
1 & 1 
\end{bmatrix}$\\ \hline
12b & $(v,\infty)(\rho)(0,1)$   & Yes &   & $\begin{bmatrix}
1 & 1 
\end{bmatrix}$\\ \hline
13 & $(0,1)(v)(\infty, \rho)$   & No &   & $\begin{bmatrix}
1 & 1/2 
\end{bmatrix}$\\ \hline

\end{tabular}
\caption{Equivalence classes of equalizing multicurves for $\Per'_4$.}
\end{table}

Cases (1)-(6) are all possible cases when the equalizing multicurve contains exactly one  curve. 
Assume that the curve $\delta$ splits the points $0,1,\infty, \rho$ and $v$ into two sets of two and three points correspondingly. First, we consider the case when $v$ belongs to the set of two points. If the second point in this set is either $0,1$ or $\rho$, then the curve is necessarily arithmetically equalizing, $f^*(\delta)=i_*(\delta)=\emptyset$. That produces cases (1)-(3).

\begin{enumerate} 
\item $P_B=(v,0)(\infty, 1,\rho)$ 
\begin{figure}[h!]
\psfrag{0}{$0$}\psfrag{1}{$1$}\psfrag{i}{$\infty$}\psfrag{g}{$\delta$}\psfrag{v}{$v$}
\psfrag{r}{$\rho$}
\includegraphics[height= 3.5cm]{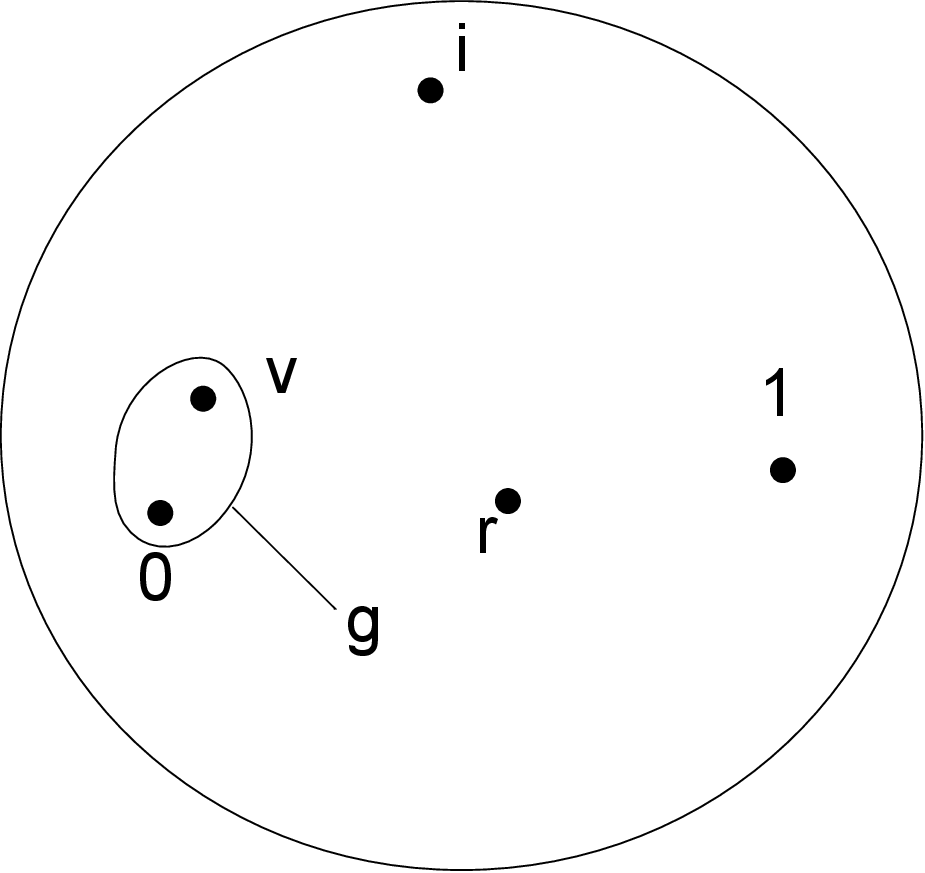}
\end{figure}


This class is images of punctures $\rho_1^0$, $\rho_2^{0}$.

\item $P_B=(v,1)(0,\infty, \rho)$ 


This class is images of punctures $\rho_1^{1}$, $\rho_2^{1}$ 

\item $P_B=(v, \rho)(0,\infty, 1)$ 


This class is images of punctures $\rho_1^{\rho}$, $\rho_2^{\rho}$ 
\noindent

\item $P_B=(v, \infty)(0,1,\rho)$ 

If the curve $\delta$ is equalizing, then it has to  split the preimages as described on the following picture:
\begin{figure}[h!]
\psfrag{0}{$0$}\psfrag{1}{$1$}\psfrag{1'}{$1'$}\psfrag{i}{$\infty$}\psfrag{i'}{$\infty'$}\psfrag{g}{$\delta$}\psfrag{v}{$v$}\psfrag{c}{ {$c$}}
\psfrag{r}{$\rho$}\psfrag{r'}{$\rho'$}\psfrag{g1}{$\delta_1$}\psfrag{g2}{$\delta_2$}

\includegraphics[height= 4cm]{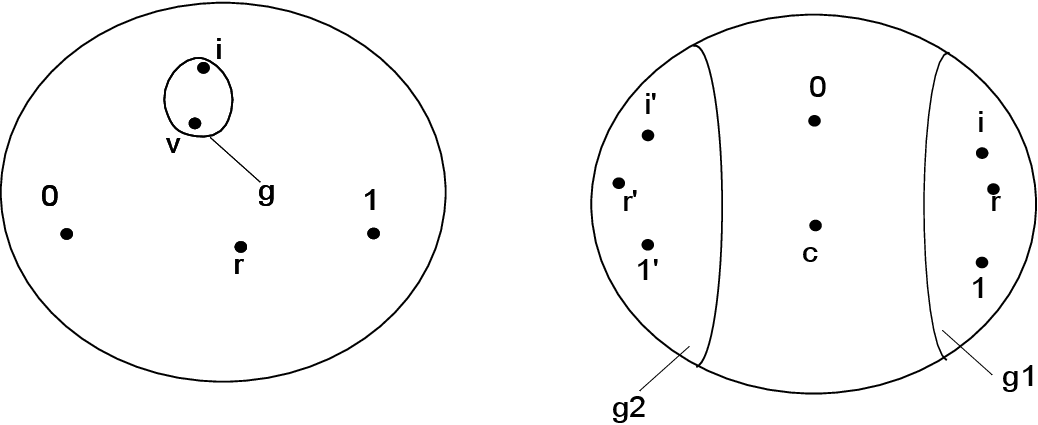}
\end{figure}
\end{enumerate}

 $f^{-1}(\delta)=\delta_1\cup \delta_2$
  
 $f^*([\delta])=i_*([\delta])=\emptyset$, $\delta_1, \delta_2$ are trivial.

 
Note that in the case $\delta$ is arithmetically equalizing 
 
This class is image of the puncture $\rho=\frac{1}{2}.$

\noindent Assume that the partition of $v$ contains two other points. These points should alternate in the cycle. Hence, we have two case $(v, 0,1)$ or $(v,\infty,\rho)$. It is convenient for us to denote the equalizing curve in the next two examples by $\gamma$.

\begin{enumerate}
\item[(5)] $P_B=(v,0,1)(\infty, \rho)$

$f^{-1}(\gamma)=\gamma_1$

$f^*([\gamma])=1/2 [\gamma_1]$, $i_*([\gamma])=[\gamma_1]$. In this case the curve $\gamma$ is topologically equalizing, but it is not arithmetically equalizing.

\begin{figure}[h!]
\psfrag{0}{$0$}\psfrag{1}{$1$}\psfrag{1'}{$1'$}\psfrag{i}{$\infty$}\psfrag{i'}{$\infty'$}\psfrag{g}{$\gamma$}\psfrag{v}{$v$}\psfrag{c}{$c$}
\psfrag{r}{$\rho$}\psfrag{r'}{$\rho'$}\psfrag{g1}{$\gamma_1$}\psfrag{g2}{$\gamma_2$}

\includegraphics[height=5cm]{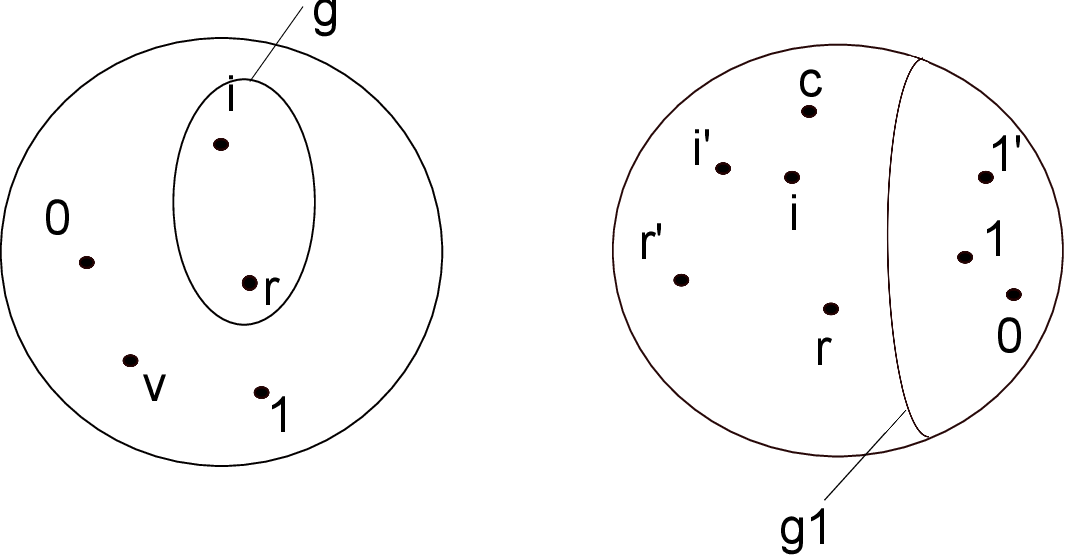}

\end{figure}

\item[(6)] $P_B=(v, \infty,\rho)(0,1)$

\begin{figure}[h!]
\psfrag{0}{$0$}\psfrag{1}{$1$}\psfrag{1'}{$1'$}\psfrag{i}{$\infty$}\psfrag{i'}{$\infty'$}\psfrag{g}{$\gamma$}\psfrag{v}{$v$}\psfrag{c}{$c$}
\psfrag{r}{$\rho$}\psfrag{r'}{$\rho'$}\psfrag{g1}{$\gamma_1$}\psfrag{g2}{$\gamma_2$}
\includegraphics[height= 4.5cm]{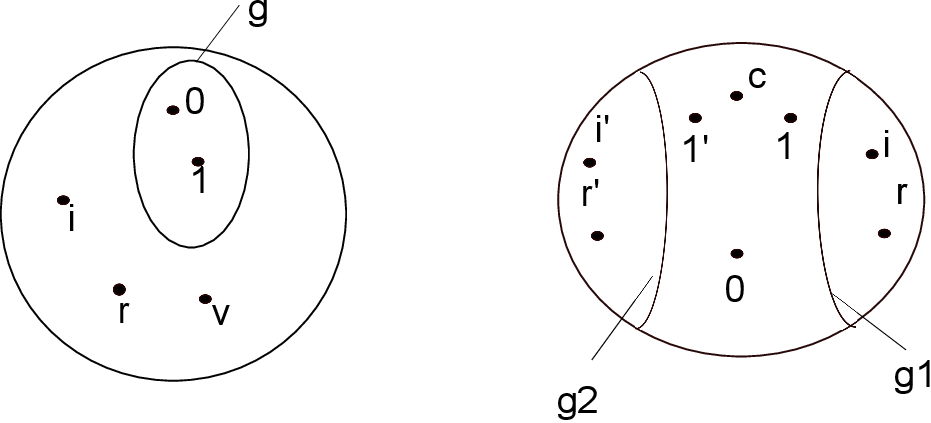}
\end{figure}

In order for $\gamma$ to be equalizing, we need to have the following picture:

 $f^{-1}(\gamma)=\gamma_1\cup \gamma_2$
 
 $f^*([\gamma])=i_*([\gamma])=[\gamma_1]$, $\gamma_2$ is peripheral.

This class is image of the puncture $\rho=\infty$.

\end{enumerate}

\noindent Now let us consider multicurves consisting of two curves, say $\gamma$ and $\delta$, which partition points $(0,1,\infty,\rho, v)$ into three groups of $2$ points inside $\gamma$, $1$ point between $\gamma$ and $\delta$,  and $2$ points inside $\delta$. First, we consider the case when the critical point $v$ belongs to a group, consisting of two points (bounded by the curve $\delta$).  If the second point inside 
$\delta$ is $0,1$ or $\rho$, $f^*([\delta])=i_*([\delta])=\emptyset$, the curve $\delta$ is itself equalizing. Hence, the curve $\gamma$ needs to be equalizing and disjoint from
$\delta$. This implies that $\gamma$ should be of type $(5)$ or $(6)$. That leaves us with the following possible configurations.

\begin{enumerate}

\item[(7)] $P_B=(v,\rho)(\infty)(0,1)$ 

The curve $\delta$ bounds $(v,\rho)$, and $\gamma$ bounds $(0,1)$.  Then $\delta$ has to be of type (3), and $\gamma$ of type (6). This pair of curves is arithmetically equalizing.


\item[(8)]  $P_B=(v,1)(0)(\infty,\rho)$

The curve $\delta$ bounds $(v,1)$, and $\gamma$ bounds $(\infty, \rho)$. Then $\delta$ has to be of type (2), and $\gamma$ of type (5).  This pair of curves is equalizing, but not arithmetically equalizing.


\item[(9)] $P_B=(v,0)(1)(\infty,\rho)$ 

The curve $\delta$ bounds $(v,0)$, and $\gamma$ bounds $(\infty, \rho)$. Then $\delta$ has to be of type (1), and $\gamma$ of type (5). This pair of curves is equalizing, but not arithmetically equalizing.


\end{enumerate}

\noindent If the second point is $\infty$, then the set consisting of one point is either  $\{0\}$, $\{1\}$, $\{\rho\}$.

\begin{enumerate}

\item[(10)] $P_B=(v,\infty)(0)(1,\rho)$. 

The curve $\gamma$ bounds $(v,\infty)$, the curve $\delta$ bounds $(1,\rho)$.

\begin{figure}[h!]
\psfrag{0}{$0$}\psfrag{1}{$1$}\psfrag{1'}{$1'$}\psfrag{i}{$\infty$}\psfrag{i'}{$\infty'$}\psfrag{d}{$\delta$}\psfrag{v}{$v$}\psfrag{c}{$c$}
\psfrag{r}{$\rho$}\psfrag{r'}{$\rho'$}\psfrag{d1}{$\delta_1$}\psfrag{g}{$\gamma$}\psfrag{g1}{$\gamma_1$}
\psfrag{d2}{$\delta_2$} \psfrag{g2}{$\gamma_2$}
\includegraphics[height= 4cm]{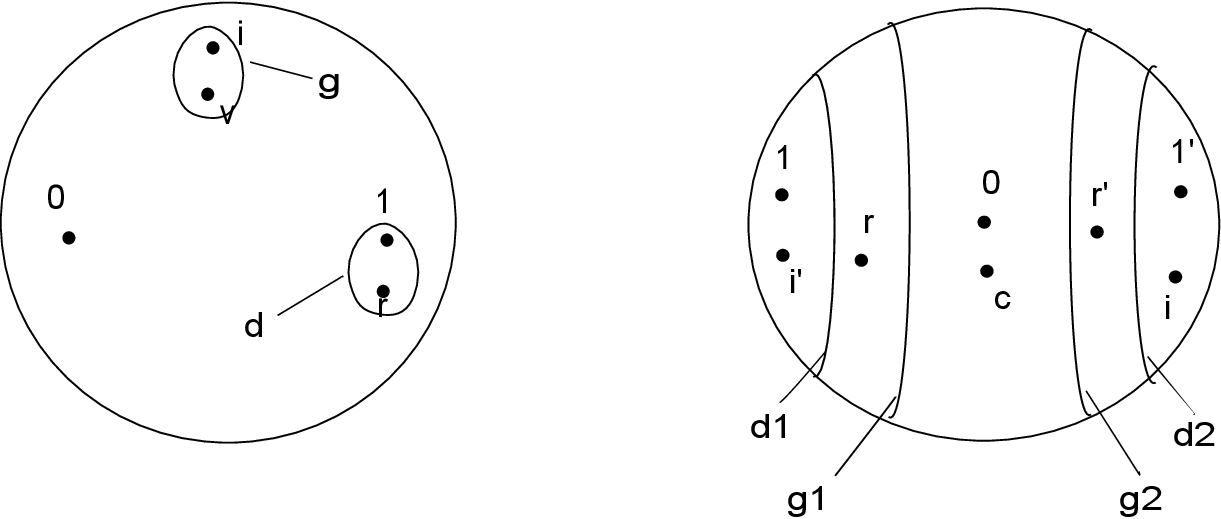}
\end{figure}

$f^{-1}(\gamma)=\gamma_1\cup \gamma_2$,$f^{-1}(\delta)=\delta_1\cup \delta_2$

 $\gamma_2$, $\delta_1$ and $\delta_2$ are peripheral.

$f^*([\gamma])=[\gamma_1]$, $i_*([\gamma])=\emptyset$, $f^*([\delta])=\emptyset$, $i_*([\delta])=[\gamma_1]$

In the bases $(\gamma,\delta)$ and $\gamma_1$, Thurston's matrix has the form
$T=(1,0)$ and forgetful matrix has the form $A=(0,1)$

The multicurve $(\gamma,\delta)$ is arithmetically equalizing.

This equalizing class is the image of the puncture 1.
\item[(11)] $P_B=(v,\infty)(1)(0,\rho)$. 

The curve $\gamma$ bounds $(v,\infty)$, the curve $\delta$ bounds $(0,\rho)$.

\begin{figure}[h!]
\psfrag{0}{$0$}\psfrag{1}{$1$}\psfrag{1'}{$1'$}\psfrag{i}{$\infty$}\psfrag{i'}{$\infty'$}\psfrag{d}{$\delta$}\psfrag{v}{$v$}\psfrag{c}{$c$}
\psfrag{r}{$\rho$}\psfrag{r'}{$\rho'$}\psfrag{d1}{$\delta_1$}\psfrag{g}{$\gamma$}\psfrag{g1}{$\gamma_1$}
\psfrag{d2}{$\delta_2$} \psfrag{g2}{$\gamma_2$}
\includegraphics[height= 4cm]{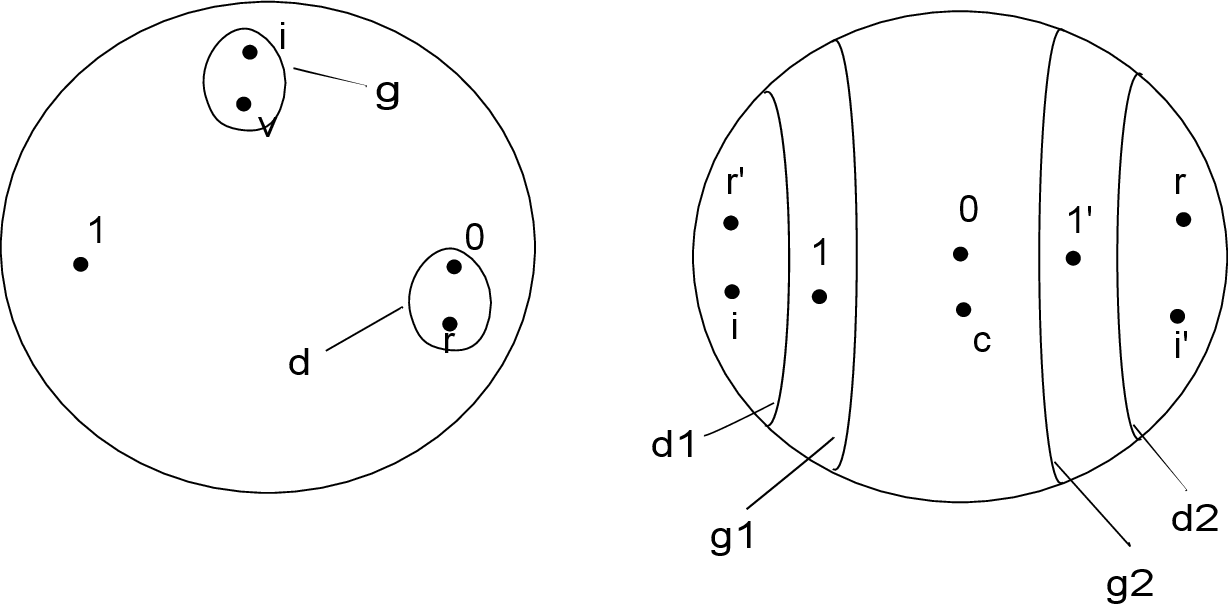}
\end{figure}

$f^-1(\gamma)=\gamma_1\cup \gamma_2$,$f^{-1}(\delta)=\delta_1\cup \delta_2$

$\gamma_2$, $\delta_1$ and $\delta_2$ are peripheral.

$f^*([\gamma])=[\gamma_1]$, $i_*([\gamma])=\emptyset$, $f^*([\delta])=\emptyset$, $i_*([\delta])=[\gamma_1]$

In the bases $(\gamma,\delta)$ and $\gamma_1$, Thurston's matrix has the form
$T=[1,0]$ and forgetful matrix has the form $A=[0,1]$

The multicurve $(\gamma,\delta)$ is arithmetically equalizing.

This equivalence class is images of the puncture  0.

\end{enumerate}

In the next case $P_B=(v,\infty)(\rho)(0,1)$, we get two different equivalence classes:

\begin{enumerate}
\item[(12a)] $P_B=(v,\infty)(\rho)(0,1)$ 

The curve $\delta$ bounds $(v, \infty)$, $\gamma$ bounds $(0,1)$.

\begin{figure}[h!]
\psfrag{0}{$0$}\psfrag{1}{$1$}\psfrag{1'}{$1'$}\psfrag{i}{$\infty$}\psfrag{i'}{$\infty'$}\psfrag{d}{$\delta$}\psfrag{v}{$v$}\psfrag{c}{$c$}
\psfrag{r}{$\rho$}\psfrag{r'}{$\rho'$}\psfrag{d1}{$\delta_1$}\psfrag{g}{$\gamma$}\psfrag{g1}{$\gamma_1$}
\psfrag{d2}{$\delta_2$} \psfrag{g2}{$\gamma_2$}
\includegraphics[height= 4cm]{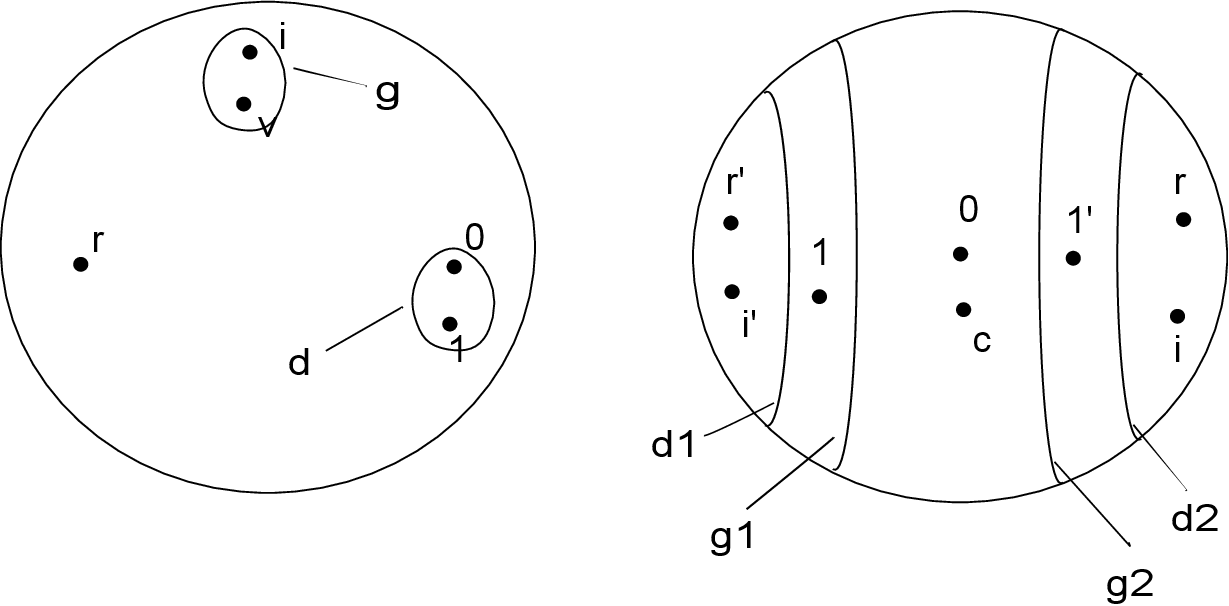}
\end{figure}

$f^{-1}(\gamma)=\gamma_1\cup \gamma_2$,$f^{-1}(\delta)=\delta_1\cup \delta_2$ 

$\gamma_1$, $\delta_1$ are peripheral, $[\gamma_2]=[\delta_2]$

$f^*([\gamma])=[\gamma_2]$, $i_*([\gamma])=\emptyset$, $f^*([\delta])=[\gamma_2],$ $i_*([\delta])=[\gamma_2]$. 

In the bases $(\gamma,\delta)$ and $\gamma_2$, Thurston's matrix has the form $(1,1)$ and 
projection matrix has the form $(0,1)$. The curve $\delta$ is arithmetically equalizing. The multicurve $(\gamma,\delta)$ is equalizing, but not arithmetically equalizing.

\item[(12b)] $P_B=(v,\infty)(\rho)(0,1)$

Both $\gamma$ and $\delta$ are are arithmetically equalizing, $\delta$ is of type (4), $\gamma$ is of type (6). 

\end{enumerate}
\noindent The last remaining case is when the point $v$ belongs to the group with one $1$ point. Then the remaining points should group in alternating order $(0,1)$ and $(\rho,\infty)$.

\begin{enumerate}

\item[(13)] $P_B= (0,1)(v)(\infty, \rho)$
\begin{figure}[h!]
\psfrag{0}{$0$}\psfrag{1}{$1$}\psfrag{1'}{$1'$}\psfrag{i}{$\infty$}\psfrag{i'}{$\infty'$}\psfrag{d}{$\delta$}\psfrag{v}{$v$}\psfrag{c}{$c$}
\psfrag{r}{$\rho$}\psfrag{r'}{$\rho'$}\psfrag{d1}{$\delta_1$}\psfrag{g}{$\gamma$}\psfrag{g1}{$\gamma_1$}
\psfrag{d2}{$\delta_1$} \psfrag{g2}{$\gamma_2$}
\includegraphics[height= 4cm]{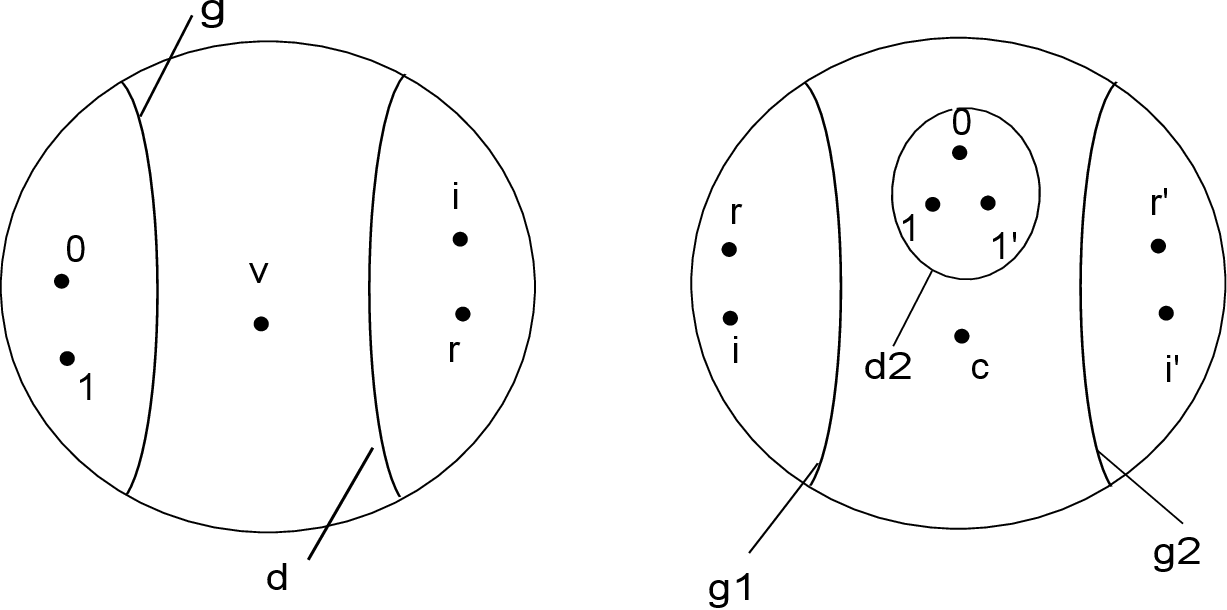}
\end{figure}

$f^{-1}(\gamma)=\gamma_1\cup \gamma_2$,$f^{-1}(\delta)=\delta_1$ 

$\gamma_2$, is peripheral, $[\gamma_1]=[\delta_1]$

$f^*([\gamma])=[\gamma_1]$, $i_*([\gamma])=[\gamma_1]$, $f^*([\delta])=1/2[\gamma_1],$ $i_*([\delta])=[\gamma_1]$. 

Thurston matrix has the form $(1,1/2)$, forgetful martrix has the form $(1,1)$. The curve $\gamma$ is arithmetically equalizing. The multicurve $(\gamma,\delta)$ is equalizing, but not arithmetically equalizing.

\end{enumerate}

\bibliographystyle{amsalpha}
\bibliography{biblio}

\end{document}